\documentclass[reqno]{amsart}

\usepackage{amsmath, amssymb, amsthm, enumerate}
\usepackage[colorlinks, linkcolor=blue, citecolor=magenta, hypertexnames=false,backref]{hyperref} 
\usepackage[dvipsnames]{xcolor}

\numberwithin{equation}{section}

\usepackage{graphicx,tikz}
\usepackage{bbm}
\newtheorem{theorem}{Theorem}[section]

\newtheorem{proposition}{Proposition}[section]
\newtheorem{corollary}{Corollary}[section]
\newtheorem*{conjecture}{Conjecture}
\newtheorem{lemma}{Lemma}[section]

\theoremstyle{definition}

\theoremstyle{remark}
\newtheorem{remark}[theorem]{Remark}

\DeclareMathOperator{\sgn}{\operatorname{sgn}}
\DeclareMathOperator*{\argmin}{\operatorname{argmin}}
\DeclareMathOperator{\HyperF}{F}

\newcommand{\Hypergeom}[5]{{\sideset{_#1}{_#2}\HyperF\!\left(
      \genfrac{}{}{0pt}{0}{#3}{#4}
      \middle|\,#5\right)}}

\begin{document}

\title[]{Polarization and greedy energy on the sphere}

\author[]{Dmitriy Bilyk \and Michelle Mastrianni \and \\ Ryan W. Matzke \and Stefan Steinerberger
}
\address{School of Mathematics, University of Minnesota, Minneapolis, MN 55455, USA}
\email{dbilyk@umn.edu}

\address{School of Mathematics, University of Minnesota, Minneapolis, MN 55455, USA}
\email{michmast@umn.edu}

\address{Department of Mathematics, Vanderbilt University, Nashville, TN 37212, USA}
\email{ryan.w.matzke@vanderbilt.edu}

\address{Department of Mathematics, University of Washington, Seattle, WA 98195, USA}
\email{steinerb@uw.edu}

\begin{abstract} 
We investigate the behavior of  a greedy sequence on the sphere $\mathbb{S}^d$ defined so that at each step  the point that minimizes the Riesz $s$-energy is added to the existing set of points.  We show  that  for $0<s<d$, the greedy sequence achieves optimal second-order behavior for the  Riesz $s$-energy (up to constants). In order to obtain this result, we prove that the second-order term of the maximal polarization with Riesz $s$-kernels is of order $N^{s/d}$ in the same range $0<s<d$. Furthermore, using the Stolarsky principle relating the $L^2$-discrepancy of a point set with the pairwise sum of distances (Riesz energy with $s=-1$), we also obtain a simple upper bound on the $L^2$-spherical cap discrepancy of  the greedy sequence    and give numerical examples that indicate that the true discrepancy is much lower.

\end{abstract}

\maketitle

\section{Introduction and background}
We consider the classical problem of distributing points on the sphere $\mathbb{S}^d$ as regularly as possible. Any such discussion requires a notion of regularity. We recall several such notions below. 

\subsection{Energy and polarization} 
Let  $\omega_{N,d} = \left\{x_1, \dots, x_N\right\} \subset \mathbb{S}^d$ denote a set of $N$ points in the sphere. 
For a symmetric, lower semi-continuous kernel $K: \mathbb{S}^d \times \mathbb{S}^d \rightarrow (- \infty, \infty]$ and a point set $\omega_{N,d}$, we define the discrete energy as
\begin{equation}\label{eq:en}
E_K(\omega_{N,d}) = 2 \sum_{1 \leq i < j \leq N } K(x_i, x_j).
\end{equation}
Of particular importance are the Riesz $s$-kernels:
\begin{equation*}
K_s(x,y) : = \begin{cases} 
\sgn(s) \|x - y\|^{-s} & s \neq 0 \\
- \log(\| x-y\|) & s = 0.
\end{cases}
\end{equation*}
Another quantity closely related to energy is polarization, defined as
\begin{equation}\label{eq:pol}
P_K(\omega_{N,d}) = \min_{x \in \mathbb{S}^d} \sum_{j=1}^{N} K(x, x_j).
\end{equation}
For a fixed $N$, we define the minimal energy and maximal (constrained) polarization, respectively, by
\begin{equation*}
\mathcal{E}_K(N) = \min_{\omega_{N,d} \subset \mathbb{S}^d} E_K(\omega_{N,d}) \quad \textup{and} \quad \mathcal{P}_K(N) = \max_{\omega_{N,d} \subset \mathbb{S}^d} P_K(\omega_{N,d}).
\end{equation*}
Both energy (see e.g. \cite{BelMO, BetS, BraG, GraS, HarS}) and polarization (see e.g. \cite{BDHSS22, FarN, FarR, Oht, RezSVo, RezSVl, Sim}), especially in the case of  Riesz kernels, have been actively studied over the past few decades: we  refer the reader to the excellent recent book \cite{BHS} for a comprehensive exposition.
The second-order asymptotic behavior for the optimal Riesz $s$-energy is well understood (see the discussion in Section \ref{sec:lower}) for $-2<s<d$.  However, results for second-order asymptotics of maximal polarization with Riesz $s$-kernels have not been obtained before (except when $d=1$). We prove these optimal bounds in Theorem \ref{thm:Singular Riesz Polarization} for the  singular potential-theoretic range $0<s<d$ and then use this result to demonstrate that, in the same range, the greedy sequence obtains optimal second-order asymptotic behavior (up to constants) for Riesz $s$-energy (Theorem \ref{thm:greedyenergy}) and, for infinitely many $N$, polarization (Corollary \ref{cor:pol}).

In this paper we only consider the range $-2 < s < d$. This is natural  for the Riesz $s$-energy on the $d$-dimensional manifold $\mathbb S^d$, since for $s<-2$ minimizing energy no longer leads to uniform distribution (indeed, optimal configurations are achieved by placing half of the points at each of  two opposite poles \cite{Bj}), while for $s>d$ the kernels $K_s$ are hypersingular. 

\subsection{Spherical cap discrepancy}  
For any $w \in \mathbb{S}^d$ and any $t \in (-1,1)$, we define
$$ C(w,t) = \left\{x \in \mathbb{S}^d: \left\langle w, x \right\rangle \geq t \right\} \subset \mathbb{S}^d$$
to be a spherical cap. The spherical cap discrepancy of a set $\omega_{N,d} = \left\{x_1, \dots, x_N\right\} \subset \mathbb{S}^d$ is then defined as
$$ D(\omega_{N,d}) = \sup_{w \in \mathbb{S}^d} \sup_{-1 < t < 1} \left| \frac{1}{N} \sum_{i=1}^{N} \chi_{C(w,t)}(x_i) - \sigma_d(C(w,t)) \right|,$$
where $\sigma_d$ is the  surface measure on $\mathbb{S}^d$ normalized so that $\sigma_d(\mathbb{S}^d) = 1$.
A fundamental result of Beck \cite{beck1, beck2} 
states that 
\begin{equation}\label{eq:beck}
N^{-\frac{1}{2}-\frac{1}{2d}} \lesssim \inf_{\omega_{N,d}} D(\omega_{N,d}) \lesssim N^{-\frac{1}{2}-\frac{1}{2d}} \sqrt{\log{N}},
\end{equation}
where the implicit constants depend on the dimension $d$. The upper bound  in \eqref{eq:beck}  follows from a probabilistic construction (jittered sampling). There are  many probabilistic  \cite{ABD, AZ, ABS, bellhouse, BE, BBL, BRSSWW, FHM} and deterministic   \cite{ABD,etayo, fer,FHM,HEAL, kor,alex, marzo, narco,s} constructions that have been investigated. Most of these explicit point sets have only been proven to have discrepancy of the order no better than  $N^{-1/2}$, although there is numerical evidence that some of them might be almost optimal. 
Instead of the spherical cap discrepancy, we will in this paper work with the $L^2-$spherical cap discrepancy
\begin{equation}\label{eq:L2cap}
 D_{L^2, \tiny \mbox{cap}}^2(\omega_{N,d}) = \int_{-1}^1 \int_{\mathbb{S}^d}  \left| \frac{1}{N} \sum_{i=1}^{N} \chi_{C(w,t)}(x_i) - \sigma_d(C(w,t)) \right|^2 d\sigma_d(w)  dt.
 \end{equation}
The $L^2$-spherical cap discrepancy is smaller than the spherical cap discrepancy though the difference is, for many examples, at most logarithmic. The $L^2$-discrepancy is connected to energy, particularly to the Riesz energy with $s=-1$ (sum of distances)  via the Stolarsky invariance principle \cite{stol}, which we state in \S \ref{sec:thm1}.

\section{Greedy sequences and main results}\label{s:main}

One of the main purposes of this paper is to discuss a simple greedy construction of sequences that turns out to simultaneously 
\begin{enumerate}[(i)]
\item achieve optimal Riesz $s$-energy, up to the second-order behavior,  when  $0 < s < d$, 
\item satisfy good distribution properties in the sense of $ D_{L^2, \tiny \mbox{cap}}^2$.
\end{enumerate}
Greedy sequences  have been studied in the context of energy \cite{brown, HarRSV, Lop, LM21, LopM, LopS, LopW, Pri1, Pri2,  SafT, Sic, wolf} and discrepancy \cite{brown, kritzinger, stein1, stein2, stein3, stein4}. 
It is well understood that greedy sequences on $\mathbb{S}^1$ tend to be fairly structured \cite{P20}, but this cannot be expected in higher dimensions. However, we show that greedy sequences on $\mathbb{S}^d$ exhibit an unusual amount of regularity: they produce nearly optimal Riesz energy for $0<s<d$ and their $L^2$-spherical cap discrepancy  satisfies favorable bounds (while numerics in \S \ref{sec:numerics} suggests they might even be optimal).

Generally, constructing a {\emph{sequence}} with good discrepancy or energy properties is significantly harder than constructing point sets for arbitrary $N$ since one is forced to keep previously chosen points.  In some situations (e.g.  discrepancy with respect to axis-parallel boxes on the torus \cite{KN, Pil}) even the optimal discrepancy bounds are slightly different for sequences and point sets. Most known constructions  of good point distributions on the sphere are not sequences. This makes the greedy {\emph{sequence}} even more valuable.

\subsection{Construction of the greedy sequence} \label{subsec:construction} We will work with an infinite sequence $(x_n)_{n=1}^{\infty}$ instead of a single set. This is more difficult since points, once placed, remain fixed.  We start with a completely arbitrary initial set of points $\left\{x_1, \dots, x_m \right\} \subset \mathbb{S}^d$. The construction is then greedy: at each step we add the point that minimizes the Riesz $s$-energy with respect to the existing set of points.  More formally, given a set $\omega_{N,d} = \left\{x_1, \dots, x_N\right\} \subset \mathbb{S}^d$, we pick
\begin{equation}\label{eq:greedy}
 x_{N+1} = \argmin_{x \in \mathbb{S}^d} \sum_{i=1}^{n} K_s(x, x_i) = \argmin_{x\in \mathbb{S}^d} E_{K_s} \big( \omega_{N,d}  \cup \{x\}\big). 
 \end{equation}
 In other words, $x_{N+1}$ is just the point where the polarization is achieved for the set $\{x_1,\dots, x_N\}$, i.e.  $$  \sum_{j=1}^{N} K_s(x_{N+1}, x_j)  = P_{K_s} \big( \{x_1,\dots, x_N\} \big).  $$
If the minimum is not assumed at a unique point, then any such point may be chosen.
Note that if $s = -1$, this construction is equivalent to maximizing the  sum of (Euclidean) distances to the points in the existing set.  When $s=0$,  i.e. $K_0(x,y) = - \log(\|x-y\|)$, {and $d = 1$ (or more generally, on any compact $\Omega \subset \mathbb{C}$)}, such sequences are called Leja sequences  in recognition of the work of Leja \cite{Lej} although they were introduced by Edrei \cite{Erd} earlier. Simultaneously, G\'{o}rski studied the case of $s = -1$ (the Newtonian kernel) for compact sets $\Omega \subset \mathbb{R}^3$ \cite{Gor}. For $s = d-2$ (i.e. the Green kernel) and compact $\Omega \subset \mathbb{R}^d$, the corresponding greedy sequences are known as Leja--G\'{o}rski sequences, which have been studied in \cite{Got01, Pri2}.   Leja points have applications in Stochastic Analysis \cite{NarJ} and Approximation/Interpolation Theory \cite{BiaCC, CalVM11, CalVM12, Chk, DeM, JanWZ, Rei, TayT}, and various numerical methods have been developed to approximate  such points (see, e.g. \cite{BagCR, BosMSV, CorD}).

As a general remark on the behavior of greedy sequences in any dimension, we note that  greedy sequences defined above (when initialized with a single point) have the property that every other point in the sequence is the antipodal point to the last one placed. This is a generalization of \cite[Theorem 2.1]{LopM}, the proof is identical. 

\begin{proposition}\label{prop:Greedy Symmetry}
Let $s \in \mathbb{R}$. Let $(x_n)_{n=1}^{\infty}$ be a greedy sequence on $\mathbb{S}^d$ for some initial point $x_1 \in \mathbb{S}^d$. Then for all $k \in \mathbb{N}$, $x_{2k} = - x_{2k-1}$.
\end{proposition}

More generally, this result holds if we replace $K_s$ with any kernel $K(x,y) = f(\|x-y\|)$ with $f$ strictly decreasing.

\subsection{Optimal second-order polarization estimates}

For $s<d$, we note that the kernel $K_s(x,y)$ is integrable on $\mathbb S^d$ in each variable and depends only on the distance between $x$ and $y$. This suggests defining the constant
\begin{equation}\label{eq:Continuous Min Energy}
 \mathcal{I}_{s,d} = \int_{\mathbb{S}^d} \int_{\mathbb{S}^d} K_s (x,y) d \sigma_d(x) d\sigma_d(y) = \int_{\mathbb{S}^d} K_s(x,y) d\sigma_d(y) , \quad \forall~x \in \mathbb{S}^{d}. \quad
\end{equation}

with $\sigma_d$ denoting, as before, the normalized uniform measure on $\mathbb{S}^d$. Observe that the constant $\mathcal{I}_{s,d} $ is negative for   $s<0$.

One would expect that polarization is maximized when the points are distributed as uniformly  as possible. In that case, one could replace summation by integration in \eqref{eq:pol} and expect the maximal Riesz polarization $\mathcal{P}_{K_s}$ to behave roughly as $\mathcal{I}_{s,d}  N$, and this is indeed correct, see e.g. \cite{LopS}. 
We prove more delicate polarization estimates with optimal (up to constants) second-order terms for $0<s<d$. 

\begin{theorem}\label{thm:Singular Riesz Polarization}
For $-2 < s < d$, there exist positive constants $b_{s,d}, c_{s,d}$ such that the following hold for all $N \in \mathbb{N}$.

\noindent For $0 < s <d$,
\begin{equation}\label{eq:opt_polarization}
-b_{s,d} N^{\frac{s}{d}} \leq   \mathcal{P}_{K_s}(N)  -  \mathcal{I}_{s,d} N  \leq  -c_{s,d} N^{\frac{s}{d}}.
\end{equation}
For $-2 < s < 0$,
\begin{equation}\label{eq:opt_polarization1}
\mathcal{I}_{s,d}   + b_{s,d} N^{\frac{s}{d}} \leq  \mathcal{P}_{K_s}(N)  -  \mathcal{I}_{s,d} N  \leq    -c_{s,d} N^{\frac{s}{d}}  .
\end{equation}
For $s=0$, 
\begin{equation}\label{eq:opt_polarization2}
- \frac{\log N}{d} - b_{0,d}   \leq    \mathcal{P}_{K_0}(N)  -  \mathcal I_{0,d} N  \leq  - c_{0,d}. 
\end{equation}
\end{theorem}

The lower bounds in the Theorem above  (proved in \S \ref{sec:lower}) follow from known optimal energy estimates (see Theorem \ref{thm:Riesz Energy Asymptotics}), but the upper bounds are novel and rely on the estimates of the generalized $L^2$ discrepancy \cite{BD, BDM, BMV}, see \S\ref{sec:upper}. 

Note that even in the  previously studied case of the circle ($d=1$), our result is new when $-2<s<-1$.  This in fact leads us to the following characterization of asymptotics for polarization on the circle:
\begin{corollary}\label{cor:Polarization circle}
For $-1 \leq s < 1$, $s \neq 0$,
\begin{equation}\label{eq:Pol Circ s geq -1}
\mathcal{P}_{K_s}(N) = \mathcal{I}_{s,1} N + \frac{2 \sgn(s) (2^s-1)}{(2 \pi)^s}\zeta(s) N^s +  \mathcal{O}( N^{s-2}).
\end{equation}
For $s = 0$,
\begin{equation}\label{eq:Pol Circ log}
\mathcal{P}_{K_0}(N) =  - \log(2).
\end{equation}
When $-2 < s <-1$, there are $b_{s,1}, c_{s,1} \in \mathbb{R}_{>0}$ so that for $N$ sufficiently large,
\begin{equation}\label{eq:Pol Circ s leq -1}
\mathcal{I}_{s,1} N - b_{s,1} N^{s} \leq \mathcal{P}_{K_s}(N) \leq  \mathcal{I}_{s,1} N - c_{s,1} N^{s}.
\end{equation}
\end{corollary}

\noindent An extensive discussion of the one-dimensional case is presented in  \S\ref{sec:circle}.

As suggested  by these sharp bounds in the one-dimensional case, as well as the asymptotics for the hypersingular case $s \geq d$, which were studied in \cite{AmbBE, BorB, BorHRS, ErdS, HKS13, HarPS}, we conjecture that the upper bounds for polarization given in Theorem \ref{thm:Singular Riesz Polarization} are  sharp  for $-2<s\le 0$ in all dimensions $d\ge 1$, i.e. the second term is of the order $\mathcal O (N^{s/d})$.

\subsection{Riesz energy of greedy sequences}
While Theorem \ref{thm:Singular Riesz Polarization} is interesting in its own right, we also use it as a tool to obtain second-order energy estimates for the greedy sequences (see  \S\ref{subsec:greedy}), which are of optimal order in the range $0<s<d$.

\begin{theorem}
\label{thm:greedyenergy}
Let $d \geq 2$, $-2 < s < d$, and let $(x_n)_{n \in\mathbb{N}}$ be a greedy sequence for $K_s$ as defined in \eqref{eq:greedy} with fixed initial point $x_1 \in \mathbb{S}^d$.  For $N>1$, denote the first $N$ elements of the sequence by $\omega_{N,d} = \{ x_1, \dots, x_N \}$. Then there exist positive constants $C_{s,d}, C'_{s,d}$  such that the following holds for $N \geq 2$.

\noindent For $0 < s < d$, 
\begin{equation}\label{eq:greedyenergy}
 - C_{s,d} N^{1 +\frac{s}{d}} \leq E_{K_s}(\omega_{N,d})  - \mathcal{I}_{s,d} N^2  \leq  - C'_{s,d} N^{1 +\frac{s}{d}}.
\end{equation}
For $-2 < s < 0$, 
\begin{equation}\label{eq:greedyenergy1}
 C_{s,d} N^{1 +\frac{s}{d}}    \leq    E_{K_s}(\omega_{N,d})  - \mathcal{I}_{s,d} N^2 \leq  -  \mathcal{I}_{s,d} N  - C_{s,d}' N^{1 +\frac{s}{d}} .
\end{equation}
For $s = 0$, 
\begin{equation}\label{eq:greedyenergy2}
- \frac{1}{d} N \log(N) -  C'_{0,d} N \leq  E_{K_0}(\omega_{N,d})  - \mathcal{I}_{0,d} N^2  \leq \mathcal{O}( N).
\end{equation}
\end{theorem}

The results of Theorem \ref{thm:greedyenergy} are new for all $d\ge 2$. {The case of $s \leq -2$ was entirely handled in \cite[Theorems 4.1 and 5.2]{LopM} and the case of $s \geq d$ has been studied in  \cite{LopS, LopW}.}
In the case of the circle $\mathbb S^1$, i.e. $d=1$, more precise information about the Riesz energy of greedy sequences is known, see Theorem \ref{thm:Energy of Greedy points Circle}.

The lower bounds in Theorem \ref{thm:greedyenergy}  are well-known optimal estimates for minimal Riesz $s$-energies discussed in \S \ref{sec:lower}, while the upper bounds for greedy sequences are the case $m =1$ of the following more general result about greedy sequences with an arbitrary number of initial points, which we prove in \S \ref{subsec:greedy}.

\begin{theorem}\label{thm:Greedy with m starting points}
Let $d \in \mathbb{N}$, $-2 < s < d$, and let $(x_n)_{n \in\mathbb{N}}$ be the  greedy sequence of points on $\mathbb{S}^d$ with respect to $K_s$ as defined in \eqref{eq:greedy} with fixed arbitrary $x_1, ..., x_m \in \mathbb{S}^d$. Let $\omega_{N,d}$ denote the set of the first $N$ points in this sequence. Then there exists a positive constant $C_{s,d}$ such that for $N > m$, the following hold.

\noindent For $ 0 \leq s < d$,
\begin{equation}
E_{K_s}(\omega_{N,d}) \leq  \mathcal{I}_{s,d}(N^2 - m^2) - C_{s,d}' (N^{1 + \frac{s}{d}} - m^{1 + \frac{s}{d}} ) - \mathcal{I}_{s,d}(N-m) + E_{K_s} (\omega_{ m,d}) + \mathcal{O}(N^{\frac{s}{d}}).
\end{equation}
For $ - 2< s \leq 0$ if $d \geq 2$ and $-1 < s \leq 0$ if $d=1$,
\begin{equation}\label{eq:greedyenergy1m}
E_{K_s}(\omega_{N,d}) \leq \mathcal{I}_{s,d}(N^2 - m^2) - \mathcal{I}_{s,d}(N-m)   - C_{s,d}' (N^{1 + \frac{s}{d}} - m^{1 + \frac{s}{d}} ) + E_{K_s} (\omega_{ m,d}) +  \mathcal{O}(1).
\end{equation}
For $d=1$ and $s=-1$,  
\begin{equation}
E_{K_{-1}}(\omega_{N,1}) \leq \mathcal{I}_{-1,1}(N^2 - m^2) - \mathcal{I}_{-1,1}(N-m) - C_{-1,1}' (\log(N) - \log(m))  + E_{K_{-1}} (\omega_{ m,1}) +  \mathcal{O}(1).
\end{equation}
For $d=1$ and $-2 < s < -1$,
\begin{equation}
E_{K_s}(\omega_{N,1}) \leq \mathcal{I}_{s,1}(N^2 - m^2) - \mathcal{I}_{s,1}(N-m)   +  E_{K_s} (\omega_{ m,1}) + \mathcal{O}(N^{1+ s}) +  \mathcal{O}(m^{1+ s}).
\end{equation}
\end{theorem}

These are the first second-order upper bounds for the energy in the literature for the case when the greedy sequence is initialized with $m$ arbitrary points, rather than a single point.   The Riesz energy of greedy sequences initialized by one point has been studied before in a general setting, see e.g. \cite{Lop, LopS, Sic}, where the first-order term was obtained, along with the second-order estimates in the case $d=1$, see \cite{LopM, LopW}.   In fact, for  $m=1$ these results give much stronger bounds on the Riesz energy of greedy sequences on the circle $\mathbb S^1$, but the proofs rely strongly on the structural properties  of the greedy sequences (see Theorem \ref{thm:Energy of Greedy points Circle} and the discussion in \S \ref{sec:greedyenergyS1}), which do not generalize to $m>1$. Hence, the general  results of Theorem \ref{thm:Greedy with m starting points} are interesting (although probably not sharp for $s \in (-2,0]$) even in the one-dimensional case).

\subsection{Polarization of greedy sequences}  

Theorem \ref{thm:Greedy with m starting points} demonstrates, in particular, that for any greedy sequence with $m$ initial data points
\begin{equation}
\lim_{N \rightarrow \infty} \frac{E_{K_s}(\omega_{N,d})}{N^2} = \mathcal{I}_{s,d}.
\end{equation}
Combining this with Theorem 4.2.2 and  Corollaries 14.6.5 and 14.6.7 from \cite{BHS}, we have that any such greedy sequence is uniformly distributed and therefore yields optimal first-order asymptotics for polarization (which was shown for $m=1$ in \cite[Thm 2.1]{LopS}  {and \cite[Lemma 3.1]{Sic}}).
\begin{corollary}\label{cor:unif}
Let $d \in \mathbb{N}$, $-2 < s < d$, and $(x_n)_{n \in\mathbb{N}}$ be the  greedy sequence of points on $\mathbb{S}^d$ with respect to $K_s$ as defined in \eqref{eq:greedy} with fixed arbitrary $x_1, ..., x_m \in \mathbb{S}^d$. Let $\omega_{N,d}$ denote the set of the first $N$ points in this sequence. Then the sequence of point sets $\{ \omega_{N,d} \}_{N=1}^{\infty}$ is uniformly distributed on $\mathbb{S}^d$, and
\begin{equation}
\lim_{N \rightarrow \infty} \frac{P_{K_s}(\omega_{N,d})}{N} = \mathcal{I}_{s,d}.
\end{equation}
\end{corollary}

Moreover, Theorem \ref{thm:greedyenergy}  states that greedy sequences on $\mathbb{S}^d$ achieve Riesz energy that is asymptotically optimal to the second-order (at least up to constants) in the potential-theoretic case $0<s<d$. This may be interpreted as a way of measuring regularity of a set of points, we would expect the greedy sequence to behave fairly regularly also with respect to other notions of uniformity.  We prove that greedy points also achieve almost maximal polarization (up to the second-order term)  most of the time in the sense of asymptotic density (in particular, for infinitely many $N$).

\begin{corollary}\label{cor:pol}
Let $d \geq 2$, $0 < s < d$, and $(x_n)_{n \in\mathbb{N}}$ be a greedy sequence for $K_s$.  Let $\omega_{N,d}$ denote the set of the first $N$ points in this sequence. For every $\varepsilon > 0$, there exists $X_{s,d, \varepsilon} > 0$ such that 
$$ \# \left\{1 \leq j \leq N: P_{K_s} (\omega_{j,d}) \geq j \cdot \mathcal{I}_{s,d} - X_{s,d,\varepsilon} \cdot j^{\frac{s}{d}} \right\} \geq \left(1 - \varepsilon \right)N \qquad \mbox{as}~N \rightarrow \infty.$$
\end{corollary}
This is optimal up to the value of $X_{s,d,\varepsilon}$ since Theorem \ref{thm:Singular Riesz Polarization}
 implies the unconditional bound
$ P_{K_s} (\omega_{j,d}) \leq \mathcal{P}_{K_s}(j) \leq \mathcal{I}_{s,d}\cdot j - c_{s,d} \cdot j^{\frac{s}{d}}.$
The proof gives a quantitative description of how $X_{s,d,\varepsilon}$ depends on $s,d$ and $\varepsilon$ in terms of the constant $C_{s,d}$ arising in (\ref{eq:greedyenergy}), we refer to the proof in \S\ref{sec:cor} for details.

\subsection{A uniform $L^2$-discrepancy bound} We now concentrate on Riesz energy with  $s=-1$ or, equivalently, on the problem of selecting $\left\{x_1, \dots, x_N \right\} \subset \mathbb{S}^d$ in such a way that the sum of distances 
$\sum_{i,j = 1}^{N} \|x_i - x_j\| $  is maximized. 
The problem has particular geometric significance in light of the Stolarsky invariance principle (formally stated in \S \ref{sec:thm1}) which states that maximizing the sum of distances is the
same as minimizing the spherical $L^2-$cap discrepancy. On $\mathbb{S}^2$, there are a large number of deterministic constructions  \cite{ABD, etayo, fer, FHM, HEAL, kor, alex, marzo, narco, s}
and some of them are known to achieve a spherical cap discrepancy as small as $N^{-1/2}$. The seminal results of Beck imply that optimal constructions should be as small as $N^{-3/4}$ (and there is  numerical evidence that some of these sequences of point sets achieve it). Our contribution to this question is two-fold:
\begin{enumerate}
\item we show that a large class of recursively defined sequences, containing all greedy sequences, achieve an $L^2-$spherical cap discrepancy of order $N^{-1/2}$ with an explicit small constant, and
\item we provide some numerical evidence that greedy sequences achieve a rate that is either $N^{-3/4}$ or very close to it, see \S \ref{sec:numerics}.
\end{enumerate}
The result will apply to a broader class of sequences: we note the trivial inequality
$$ \max_{x \in \mathbb{S}^d}  \sum_{i=1}^{n} \| x_{}- x_i\| \geq n \int_{\mathbb{S}^d}  \sum_{i=1}^{n} \| x- x_i\| d\sigma_d(x) = n \cdot (-\mathcal{I}_{-1, d}),$$
where we note that $-\mathcal{I}_{-1, d} > 0$.
The subsequent theorem applies to all sequences where the next element $x_{n+1}$ is always chosen in such a way that the inequality above is satisfied, i.e. rather than maximizing the sum, we choose any point where the sum exceeds its mean value (and, in particular, the greedy constructions always satisfy the inequality). 
\begin{theorem} 
\label{thm:L2discrepancy}
Let $\omega_N = \left\{x_1, \dots, x_m \right\} \subset \mathbb{S}^d$ be an arbitrary initial set and suppose that for all $n \geq m$, the set is extended to a sequence satisfying 
$$  \sum_{i=1}^{n} \| x_{n+1}- x_i\| \geq n \cdot (-\mathcal{I}_{-1,d}).$$
Let $\omega_{N,d}$ denote the set of the first $N$ points in this sequence. Then, for any $N \geq m$, 
$$ D_{L^2, {\tiny \emph{cap}}}(\omega_{N,d}) \leq \sqrt{ \frac{1}{d} \frac{\Gamma((d+1)/2)}{\sqrt{\pi} \Gamma(d/2)}}    (-\mathcal{I}_{-1,d})^{1/2} \left( \frac{1}{N} + \frac{m^2 - m}{N^2} \right)^{1/2}.$$
On $\mathbb{S}^2$, this expression simplifies to
$$ D_{L^2, \emph{cap}}(\omega_{N,2}) \leq  \frac{\sqrt{2}}{\sqrt{3}} \left( \frac{1}{N} + \frac{m^2 - m}{N^2} \right)^{1/2}.$$
\end{theorem}

Observe that, due to the Stolarsky principle (Theorem \ref{thm:stolarsky}), when $m=1$, the right-hand side  this inequality exactly matches   the first-order term in the upper bound in \eqref{eq:greedyenergy1} in Theorem \ref{thm:greedyenergy}. In fact, inequality \eqref{eq:greedyenergy1} is even  stronger, since it includes a negative second-order term (similar conclusions for all $m\ge 1$ follow from \eqref{eq:greedyenergy1m}).  However, this theorem still provides new information since the bounds here apply to a much wider class of sequences than just purely greedy sequences.

 We first note that it applies to a fairly large family of sequences and gives a uniform $N^{-1/2}-$bound for all of them. This $N^{-1/2}-$bound is generically tight. One might be inclined to believe that the greedy sequence is better behaved, which is suggested by the case $0<s<d$ of  Theorem \ref{thm:greedyenergy}.  
 
 Theorem \ref{thm:Singular Riesz Polarization}, in the case $s=-1$,  can be rewritten as saying that there is some $c_d > 0$ such that for any $N$-point set $\{x_1, \cdots, x_N\},$
$$ \max_{x \in \mathbb{S}^d}  \sum_{i=1}^{N} \| x_{}- x_i\| \geq   N \cdot (-\mathcal{I}_{-1, d}) + \frac{c_{d}}{ N^{1/d}}.$$
While this is sharp for $d=1$ and we believe it to be sharp for higher dimensions, it cannot be sharp for a greedy sequence ``on average", which becomes evident from the following result proved in \S \ref{sec:prop:disc} (recall that  $-\mathcal{I}_{-1, d} = \frac43$). 

\begin{proposition}\label{prop:disc}  Let $(x_n)$ be a greedy sequence on $\mathbb{S}^2$ with $m\ge 1$ initial points. For $N$ sufficiently large, depending on $m$,
\begin{equation}\label{eq:averagepol}
 \frac{1}{25} \leq  \frac{1}{N}\sum_{n=1}^{N-1} \left( \max_{x \in \mathbb{S}^2}  \sum_{i=1}^{n} \| x - x_i\| -  \frac{4 n}{3} \right) \leq \frac{2}{3}.
 \end{equation} 
\end{proposition}
If greedy sequences indeed have good $L^2-$discrepancy properties (as suggested by numerics) and if $D_{L^2, \emph{cap}}(\omega_{N,2}) \ll N^{-1/2}$, then
the upper bound in Proposition \ref{prop:disc} has to be asymptotically tight. The gap between $1/25$ and $2/3$ is a quantitative way of measuring our lack of understanding of the underlying dynamics. 

The proof of Proposition \ref{prop:disc} can be easily adapted, with different constants, to other $s$-Riesz kernels and to all dimensions $d\ge 2$ (as well as for a wide range of more general kernels). We choose to state it  just for $s=-1$ and $d=2$ for the sake of simplicity of exposition. 

\begin{remark}
The $L^2$ spherical cap discrepancy can also be interpreted as the worst-case error for numeric integration (with equal weights) for the Sobolev space $H^{\frac{d+1}{2}}(\mathbb{S}^d)$, see \cite{BraD}, and our results can be translated to this setting. Greedy sequences in the context of worst-case error for numeric integration for reproducing kernel Hilbert spaces have been studied, for example, in \cite{SanKH}, where a similar bound (although with an additional optimization over weights) was obtained.
\end{remark}

\subsection{Outline}
The outline of the paper is as follows. In \S \ref{sec:circle}, we collect a number of results about both polarization and greedy sequences in the special case of the circle $\mathbb{S}^1$, presenting the proof of Corollary \ref{cor:Polarization circle} as well as the analogues of Theorem \ref{thm:greedyenergy} and Corollary \ref{cor:pol} for $d = 1$. In  \S \ref{sec:higherdim} we present the proofs of some of  the main results for $d \ge 2$:  Theorem \ref{thm:Singular Riesz Polarization} in \S \ref{sec:upper}--\ref{sec:lower}, Theorems \ref{thm:greedyenergy} and \ref{thm:Greedy with m starting points} in \S\ref{subsec:greedy}, and Corollary \ref{cor:pol} in \S\ref{sec:cor}. In \S \ref{sec:thm1} we prove the $L^2$-discrepancy bound (Theorem \ref{thm:L2discrepancy}) and Proposition \ref{prop:disc} and discuss some similar  known results. We conclude with discussion of numerical properties of greedy sequences in \S \ref{sec:numerics}, and provide some auxiliary results in \S \ref{sec:Appendix}.

\section{Polarization and Greedy Sequences on $\mathbb{S}^1$}
\label{sec:circle}

The case of $d = 1$ in Theorem \ref{thm:Singular Riesz Polarization} and Corollary \ref{thm:greedyenergy} has mostly been settled in the literature previously, as the unique properties of the circle yield convenient examples of point sets that maximize polarization and greedy energy. In this section we collect known results on the circle for polarization and greedy energy and note where our Theorem \ref{thm:Singular Riesz Polarization} and Corollary \ref{thm:greedyenergy} fill in some gaps or differ.

\subsection{Polarization on the circle and proof of Corollary \ref{cor:Polarization circle}}
 
In \cite[Theorem 1]{HKS13} it is shown that equally spaced points on $\mathbb{S}^1$, denoted $\omega_N^*$, are optimal for polarization for any kernel  $K(x,y) = f( \arccos(\langle x, y \rangle))$ such that $f$ is decreasing and convex on $[0, \pi]$. This includes the Riesz kernels for $-1 \le s < \infty$, but not $s < -1$. However, for $-2 < s < -1$, $\omega_N^*$ still has optimal asymptotic behavior, and likely maximizes polarization as well.

For the range $-2 < s$, the points that minimize the discrete potential with respect to $\omega_N^*$ are the midpoints of an arch between two $N$th roots of unity. Because the midpoints of an arch between two $N$th roots of unity are themselves $2N$th roots of unity, there is a natural expression for the polarization on $\mathbb{S}^1$ in terms of the corresponding energies for $N$ and $2N$ equally spaced points:
\begin{equation}
\label{eqn:diffofenergies}
P_{K_s}(\omega_N^*) = \frac{E_{K_s}(\omega_{2N}^*)}{2N} - \frac{E_{K_s}(\omega_N^*)}{N}.
\end{equation}

There is an explicit formula for the energy of $N$ equally spaced points on $\mathbb{S}^1$.

\begin{theorem}[Theorem 1.1 in \cite{BHS09}] 
\label{thm:zetafunction}
Let $s \in \mathbb{R}$, $s \neq 0,1,3,5,...,$ and let $q$ be any nonnegative integer such that $q \ge (s -1)/2$. If $\omega_N^{\star}$ is a configuration of $N$ equally spaced points on $\mathbb{S}^1$, then
$$E_{K_s}(\omega_N^*) = \mathcal{I}_{s,1} N^2 + \frac{2 \sgn(s)}{(2\pi)^s} \sum_{n=0}^q a_n(s) \zeta(s-2n)N^{1+s-2n}+\mathcal{O}(N^{s-1-2q}),$$
where $\zeta(s)$ is the classical Riemann zeta function and $a_n(s)$ are the coefficients in the expansion
$$\Big(\frac{\sin \pi z}{\pi z}\Big)^{-s} = \sum_{n=0}^{\infty} a_n(s) z^{2n}, \hspace{1cm} |z| < 1.$$
\end{theorem}

Combining Theorem \ref{thm:zetafunction} with \eqref{eqn:diffofenergies}, we have that on $\mathbb{S}^1$ for $-2 < s < 1$, $s \neq 0$,
\begin{equation}
\label{eqn:expansion}
P_{K_s}(\omega_N^*) = \mathcal{I}_{s,1}  N - \frac{2 \sgn(s)}{(2\pi)^s} \zeta(s) N^s (2^s-1) + \mathcal{O}(N^{s-2}).
\end{equation}

This second order term also explicitly computed in \cite[Proposition 4.1]{LM21} and \cite[Lemma 3.10]{LopM}. In \cite[page 623]{BHS09}, the authors also show that 
\begin{equation}\label{eq:Log Energy equally spaced points}
E_{K_0}(\omega_N^*) = - N \log(N)
\end{equation}
which, when combined with \eqref{eqn:diffofenergies}, gives us that
\begin{equation}\label{eqn:expansion log Pol circle}
P_{K_0}(\omega_N^*) = - \log(2).
\end{equation}

Both \eqref{eqn:expansion} and \eqref{eqn:expansion log Pol circle} yield a lower bound on optimal polarization $\mathcal{P}_{K_s}(N) \ge P_{K_s}(\omega_N^*)$,  which   matches the  upper bounds given in Theorem \ref{thm:Singular Riesz Polarization} for $-2<s<1$.  
In the case $-1 \leq s < 1$, optimality of roots of unity, \cite[Theorem 1]{HKS13} shows that $\mathcal{P}_{K_s}(N) = P_{K_s}(\omega_N^*)$, completing the proof of \eqref{eq:Pol Circ s geq -1} and \eqref{eq:Pol Circ log}.

The proof of the upper bounds in  Theorem \ref{thm:Singular Riesz Polarization} (which is presented  in \S \ref{sec:upper} and holds for all $d\ge 1$) covers  the range $-2<s<  -1$ in which the roots of unity are not known to be optimal, proving \eqref{eq:Pol Circ s leq -1}. Thus this case of Theorem \ref{thm:Singular Riesz Polarization} is new even in the one-dimensional case.  This completes the proof of polarization estimates for $d=1$, i.e. Corollary \ref{cor:Polarization circle}.

\subsection{Energy for greedy sequences}\label{sec:greedyenergyS1}

The behavior of greedy sequences on $\mathbb{S}^1$ is also well-studied. For $-2 < s$, it is known that any greedy sequence is in fact a classical van der Corput sequence \cite{vdc} (see \cite[Theorem 5]{BiaCC}, \cite[Lemma 3.7]{LopM}, \cite[Sec 1.2]{LopW}, \cite[Lemmas 4.1 and 4.2]{LopS}, and also \cite[Example 2]{wolf}, which is perhaps the earliest observation of this kind, but just for $s=-1$). A similar result was shown for a large class of kernels  (whenever $K(x,y) = f( \arccos( \langle x , y \rangle))$, and $f$ is a bounded, continuous, decreasing, convex function of the geodesic distance) in \cite[Thm 2.1]{P20}.  This has made explicit computations for bounds of the asymptotic behavior of the greedy Riesz energies possible (see Theorems 1.1, 1.2, and 1.5 in \cite{LopW} and Theorems 3.16, 3.17, 3.18 in \cite{LopM}). Here we collect the these results, in less detail.
\begin{theorem}\label{thm:Energy of Greedy points Circle}
On the circle $\mathbb{S}^1$, for $-2 < s < 1$ and $N \geq 2$,  if $\omega_N$ is the first $N$ elements of a greedy sequence, then
\begin{equation}
E_{K_s}(\omega_N) = \mathcal{I}_{s,1} N^2 + \begin{cases}
\mathcal{O}(N^{s+1}) & s \in (-1,0) \cup (0,1) \\
- N \log(N) + \mathcal{O}(N) & s = 0 \\
\mathcal{O}(\log(N)) & s = -1 \\
\mathcal{O}(1) & s \in (-2,-1)
\end{cases}.
\end{equation}
The order (of the second-order term) in each case cannot be improved.
\end{theorem}

We note that for $-2 < s \leq 0$, this is an improvement on the upper bounds achieved in Theorem \ref{thm:greedyenergy}, suggesting that there is likely room for improvement on the upper bounds in higher dimensions in this range.

Moreover, according to Theorem \ref{thm:Riesz Energy Asymptotics}, which states that the optimal second-order term for the Riesz $s$-energy on $\mathbb S^d$ is  $ \mathcal O (N^{1+\frac{s}{d}})$,  we see that for $s > -1$, point sets produced via the greedy algorithm have optimal asymptotic behavior on $\mathbb S^1$, whereas for $ -2 < s \leq -1$, the greedy algorithm \emph{does not} produce point sets with optimal Riesz energy, which leads us to the open question of whether these results also hold true for higher dimensions.

At the same time, at least when $s=-1$, these bounds are actually sharp for {\emph{sequences}}. Indeed, the case of $s=-1$ in Theorem \ref{thm:Energy of Greedy points Circle} and the Stolarsky invariance principle (Theorem \ref{thm:stolarsky})  imply  that the $L^2$-spherical caps discrepancy satisfies $D_{L^2,cap} (\omega_N) = \mathcal O (\sqrt{N^{-1} \log N})$. But spherical caps on the circle are just intervals, hence, this is just the classical periodic $L^2$-discrepancy of one-dimensional sequences,  for which  $\mathcal O (N^{-1} \sqrt{\log N})$ is known to be the  optimal order, as shown in \cite{Proi} (see also \cite{Kirk, Pil}) with the ideas going back to the seminal results of Roth \cite{Roth}. In fact, as mentioned earlier, in this case, the greedy sequence with one initial point is just  the van der Corput sequence, whose discrepancy is well studied \cite{CF, kritzinger, PA}. 

On the other hand, Theorem \ref{thm:Riesz Energy Asymptotics} shows that the optimal second-order term for the energy in the case $d=1$ and $s=-1$  is $ \mathcal O (N^{1+\frac{s}{d}}) = \mathcal O (1)$, and therefore, the  optimal $L^2$-discrepancy is  $\mathcal O (N^{-1})$ (by the Stolarsky principle), which is easily seen to be achieved by $N$ equally spaced points on the circle. This  highlights an important difference between the behavior of $N$-point sets and {\emph{infinite sequences}}.

\subsection{Polarization for greedy sequences}

As discussed earlier, the fact that any greedy sequence is uniformly distributed suggests that they may behave well for different measures of uniformity. For polarization, the asymptotic behavior of a greedy sequence for Riesz kernels on the circle was shown in \cite[Theorem 3.11]{LopM} and \cite[Theorems 1.1 and 1.4]{LM21} (the case of hypersingular Riesz kernels was also handled in \cite{LM21}  
{ and the case of $s \leq -2$ in \cite[Theorem 4.1 and 5.2]{LopM}}).

\begin{theorem}\label{thm:Polar of Greedy points Circle}
On the circle $\mathbb{S}^1$, for $-2 < s < 1$ and $N \geq 1$, if $\omega_N$ is the first $N$ elements of a greedy sequence, then
\begin{equation}
P_{K_s}(\omega_N) = \mathcal{I}_{s,1} N + \begin{cases}
\mathcal{O}(N^s) & s \in (-1,0) \cup (0,1) \\
\mathcal{O}(\log(N)) & s = 0 \\
\mathcal{O}(1) & -2 <s < 0 \\
\end{cases}.
\end{equation}
The order (of the second-order term) in each case cannot be improved.
\end{theorem}

Comparing this to Corollary \ref{cor:Polarization circle}, we see that  the greedy sequences on the circle have good second-order asymptotic behavior for $0 < s$, but not for $-2 < s \leq 0$.

\section{The Proof of Theorems \ref{thm:Singular Riesz Polarization} and   \ref{thm:greedyenergy}}
\label{sec:higherdim}

We proceed with the proof of Theorem \ref{thm:Singular Riesz Polarization} for $d \ge 2$. We associate the Riesz $s$-kernel $K_s(x,y)$ with a function of the inner product $t$ between the points $x$ and $y$ by observing that $\| x-y \| = ( 2 - 2 t)^{1/2}$ and setting $K_s (x,y) =  f_s (\langle x,y \rangle)$, where:
\begin{align*}
f_{s}(t) &= \begin{cases}
\sgn(s) (2 - 2t)^{-s/2} & s \neq 0 \\
-\frac{1}{2} \log( 2 - 2t) & s=0.
\end{cases}
\end{align*}
Further, to facilitate our proof of the upper bound for polarization, we also introduce non-singular approximations to the functions $f_s$. For $\varepsilon > 0$, take
\begin{equation}
f_{s,\varepsilon}(t) = \begin{cases}
\sgn(s) (2 + \varepsilon - 2t)^{-s/2} & s \neq 0 \\
-\frac{1}{2} \log( 2 + \varepsilon - 2t) & s=0.
\end{cases}
\end{equation}

\subsection{Proof of Theorem \ref{thm:Singular Riesz Polarization}: upper bound}\label{sec:upper}

\label{subsec:upperbound}

For $d \in \mathbb{N}$, let $w_{d}(t) = (1-t^2)^{\frac{d-2}{2}}$. For any $f \in L_{w_d}^1([-1,1])$, we have the Gegenbauer expansion
\begin{equation}\label{eq:Gegen Expand General}
f(t) \sim \sum_{n=0}^{\infty} \widehat{f}(n,d) \frac{2n+d-1}{d-1} C_n^{\frac{d-1}{2}}(t),
\end{equation}
where $C_n^\lambda$ are the standard Gegenbauer polynomials (see e.g. \cite{DX} for details) and 
\begin{equation}\label{eq:GegenCoeff}
\widehat{f}(n,d) = \frac{\Gamma(\frac{d+1}{2}) n! \Gamma(d-1)}{\sqrt{\pi} \Gamma( \frac{d}{2}) \Gamma(n+d-1)} \int_{-1}^{1} f(t) C_n^{\frac{d-1}{2}}(t) w_d(t) dt.
\end{equation}

Note that for all $s \in \mathbb{R}$ and $\varepsilon >0$, the function $f_{s,\varepsilon}$ is continuous, and therefore in $L_{w_d}^2([-1,1])$. On the other hand, $f_s \in L_{w_d}^2([-1,1])$ only when $s< \frac{d}{2}$, which would lead to certain technical complications in the proofs when $s\ge \frac{d}{2}$. At the same time $f_s \in L_{w_d}^1([-1,1])$ for all $s<d$, in other words, $K_s$ is integrable on the sphere, i.e. $\mathcal{I}_{s,d} < \infty$, for $s<d$ (this is the potential-theoretic case).

We shall need the information about the behavior of the Gegenbauer coefficients of the Riesz kernels $f_s$ as well as their approximations $f_{s,\varepsilon}$. We postpone a detailed analysis to the Appendix and only mention the most relevant information here: for $0<s<d$, the coefficients $\widehat{f}_{s,\varepsilon} (n,d) $ and $\widehat{f}_{s} (n,d) $ are non-negative and decreasing, and $\widehat{f}_{s} (n,d) $ is of the order $n^{s-d}$ (Corollary \ref{cor:Riesz Gegen Coeff}). 

We start with the case $-2 < s < \frac{d}{2}$ and will then explain the adjustments needed in the range $\frac{d}{2} \le s < d$. Starting with an arbitrary $\omega_{N,d} = \{ x_1,\dots, x_N\}$ and trivially estimating the integral by the average, we observe that 
$$ P_{K_s} (\omega_{N,d}) = \min_{x \in \mathbb{S}^d} \sum_{j=1}^{N} K_s(x, x_j) \le  \sum_{j=1}^{N} \int_{\mathbb S^d}  K_s (x, x_j) \, d\sigma_d (x) =  \mathcal{I}_{s,d} N ,  $$
where we have used that due to rotational invariance the integral  $\int_{\mathbb S^d}  K_s (x, z) \, d\sigma_d (x)
= \mathcal{I}_{s,d} $ is independent of $z\in \mathbb S^d$. 
Therefore, $$ \mathcal{I}_{s,d}  - \frac{1}{N}  P_{K_s} (\omega_{N,d})
= \max_{x\in \mathbb S^d}  \bigg( \mathcal{I}_{s,d}  - \frac{1}{N} \sum_{j=1}^{N} K_s(x, x_j) \bigg) \ge 0,$$ and we can thus estimate this expression from below by  the $L^2$-average:
\begin{align*}
\mathcal{I}_{s,d}  - \frac{1}{N}  P_{K_s} (\omega_{N,d}) &=  \max_{x\in \mathbb S^d}  \bigg( \mathcal{I}_{s,d}  - \frac{1}{N} \sum_{j=1}^{N} K_s(x, x_j) \bigg)\\
& \ge \bigg( \int_{\mathbb S^d} \bigg| \mathcal{I}_{s,d} -  \frac{1}{N} \sum_{j=1}^{N} K_s(x, x_j) \bigg|^2 d\sigma_d (x)   \bigg)^{1/2}\\
& =  \bigg( \int_{\mathbb S^d} \bigg| \int_{\mathbb S^d} f_s (\langle x,y \rangle) \, d\sigma_d (y)  -  \frac{1}{N} \sum_{j=1}^{N} f_s(\langle x, x_j \rangle) \bigg|^2 d\sigma_d (x)   \bigg)^{1/2}\\
& = D_{L^2, f_s} (\omega_{N,d}),
\end{align*}
where $D_{L^2, f} (\omega_{N,d})$ is the generalized $L^2$-discrepancy of $\omega_{N,d}$ with respect to the function $f$ which is studied in \cite{BD,BDM,BMV} and well-defined for $ f \in L^2_{w_d} ([-1,1])$. Observe the similarity of this notion to the classical $L^2$-spherical cap discrepancy \eqref{eq:L2cap}. Indeed,  taking $f = {\bf{1}}_{[t,1]}$, one obtains the inner integral in the definition \eqref{eq:L2cap} of $ D_{L^2, \tiny \mbox{cap}}(\omega_{N,d})$. 
It has been shown \cite[Theorem 4.2]{BD} that  $$  D_{L^2, f} (\omega_{N,d}) \ge C_d \min_{1\le n \le c_d N^{1/d}} | \widehat{f} (n,d) |  $$ for some constants $C_d$, $c_d >0$. 
Since $f_s \in L^2_{w_d} ([-1,1])$  for $-2 < s<d/2$ and Gegenbauer coefficients  $\widehat{f_s} (n,d)$, $n > 0$,  are  positive and decreasing (Corollary \ref{cor:Riesz Gegen Coeff}), we find that  
$$  \mathcal{I}_{s,d}N  -   P_{K_s} (\omega_{N,d}) \ge   N  D_{L^2, f_s} (\omega_{N,d}) \ge C_d N  \widehat{f}_{s}\Big( \big\lfloor c_d N^{1/d} \big\rfloor ,d \Big) \ge  c_{s,d}  N^{\frac{s}{d} },$$
where we have used the asymptotics of the Gegenbauer coefficients $\widehat{f_s} (n,d)$. Taking the supremum over $\omega_{N,d}$ proves the upper bound in  Theorem \ref{thm:Singular Riesz Polarization} for $-2< s < d/2$. \\

Turning to the case $s\ge d/2$,  we can repeat the argument above verbatim for the kernel $K_{s,\varepsilon } (x,y ) = f_{s,\varepsilon} (\langle x,y \rangle)$, using the fact that the Gegenbauer coefficients of $f_{s,\varepsilon} $ are positive and decreasing (Corollary \ref{cor:Riesz epsilon Gegen Coeff}) and the fact that $0\le f_{s,\varepsilon} (t) \le f_{s} (t) $),  to find that 
$$ \mathcal{I}_{s,d} N -   P_{K_{s,\varepsilon}} (\omega_{N,d}) \ge C_d N  \widehat{f}_{s,\varepsilon}\Big( \big\lfloor c_d N^{1/d} \big\rfloor ,d \Big) . $$ 
Taking the limit as $\varepsilon \rightarrow 0$ which is justified by Lemma \ref{lem:Limit for Polarization} and by the Lebesgue dominated convergence theorem, since $0\le f_{s,\varepsilon} (t) \le f_{s} (t) $.
Using the asymptotics of $\widehat{f_s} (n,d)$ from Corollary \ref{cor:Riesz Gegen Coeff}, one  obtains 
$$ \mathcal{I}_{s,d} N -   P_{K_s} (\omega_{N,d})   \ge C_d N  \widehat{f}_{s}\Big( \big\lfloor c_d N^{1/d} \big\rfloor ,d \Big) \ge  c_{s,d}  N^{\frac{s}{d} }, $$ which proves the required bound  for the range $ d/2 \le s < d$.

\subsection{Proof of Theorem \ref{thm:Singular Riesz Polarization}: lower bound}\label{sec:lower}

We collect some known results from which the lower bound of Theorem \ref{thm:Singular Riesz Polarization} follows immediately for $d \ge 2$. The first relates the maximal polarization to the minimal energy.

\begin{proposition}[Proposition 14.1.1, \cite{BHS}]\label{prop:Relate Polarization and Energy LB}
For all $N \geq 2$ and kernels $K$, we have
\begin{equation}
\mathcal{P}_K(N) \geq \frac{\mathcal{E}_K(N+1)}{N+1} \geq \frac{\mathcal{E}_K(N)}{N-1}.
\end{equation}
\end{proposition}

Optimal second-order asymptotics for the discrete Riesz $s$-energy in the range $-2<s<d$ have been computed by various authors \cite{Brauchart, BHS12, KS, RSZ, wagner1, wagner2}, see also Theorems 6.4.5, 6.4.6, and 6.4.7 in \cite{BHS}. These  bounds are as follows. 

\begin{theorem}\label{thm:Riesz Energy Asymptotics}
For $d \geq 1$, $-2 < s < d$, $s \neq 0$, there exist  positive constants $C_{s,d}$, $C'_{s,d}$ such that for $N \geq 2$,
\begin{align*}
\mathcal{I}_{s,d} N^2 - C'_{s,d} N^{1+\frac{s}{d}} & \le \mathcal{E}_{K_s}(N) \le \mathcal{I}_{s,d} N^2 - C_{s,d} N^{1+\frac{s}{d}}, \quad & s>0,\\
\mathcal{I}_{s,d} N^2 + C'_{s,d} N^{1+\frac{s}{d}} & \le \mathcal{E}_{K_s}(N) \le \mathcal{I}_{s,d} N^2 +  C_{s,d} N^{1+\frac{s}{d}}, \quad & s<0. 
\end{align*} 
If $s = 0$, then
\begin{equation}
\mathcal{E}_{K_s}(N) = \mathcal{I}_{s,d} N^2 - \frac{1}{d}N \log(N) + \mathcal{O}(N).
\end{equation}
\end{theorem}

Combining Theorem \ref{thm:Riesz Energy Asymptotics} and Proposition \ref{prop:Relate Polarization and Energy LB}, we deduce the following bounds.
\begin{corollary}\label{cor:Polarization asymptotics lower bound}
For $d \in \mathbb{N}$, $-2 < s < d$, $s \neq 0$, and $N \in \mathbb{N}$,
\begin{equation}
\mathcal{P}_{K_s}(N) \geq \begin{cases}
\mathcal{I}_{s,d}N - C_{s,d}' 2^{\frac{s}{d}} N^{\frac{s}{d}} + \mathcal{I}_{s,d} & 0<s<d, \\
\mathcal{I}_{s,d} N + \mathcal{I}_{s,d} + C_{s,d}' 2^{\frac{s}{d}}  N^{\frac{s}{d}} , & -2 <s < 0
\end{cases},
\end{equation}
where $C_{s,d}'$ is as in Theorem \ref{thm:Riesz Energy Asymptotics}.

If $s = 0$, then
\begin{equation}
\mathcal{P}_{K_0}(N) \geq \mathcal{I}_{0,d} N - \frac{1}{d} \log(N) + \mathcal{O}(1). 
\end{equation}
\end{corollary}

Since the constant term is essential when $s<0$, this gives the lower bound of Theorem \ref{thm:Singular Riesz Polarization} for the case $d \ge 2$.

\subsection{Greedy energy: proof of Theorems \ref{thm:greedyenergy} and \ref{thm:Greedy with m starting points}}
\label{subsec:greedy}
The lower bounds in Theorem \ref{thm:greedyenergy} are just the general lower bounds for minimal Riesz energies presented in  Theorem \ref{thm:Riesz Energy Asymptotics} above, so we concentrate on the upper bound of Theorem \ref{thm:Greedy with m starting points}.

\begin{proof}[Proof of Theorem \ref{thm:Greedy with m starting points}]
Let $\omega_{N,d}$ denote the set of the first $N$ points of the greedy sequence with respect to $K_s$  with $m$ initial points. Observe that, by construction, for $k \geq m$ 
$$ \sum_{j=1}^{k} K_s (x_{k+1}, x_j) =  P_{K_s} (\omega_{k,d}).$$ 
Therefore the discrete energy of  $\omega_{N+1,d}$ satisfies
\begin{align*}
E_{K_s} (\omega_{ N+1,d}) & = E_{K_s} (\omega_{ m,d}) + 2 \sum_{k=m}^{N} \sum_{j=1}^{k} K_s (x_{k+1}, x_j) \\
& =  E_{K_s} (\omega_{ m,d}) +  2 \sum_{k=m}^{N} P_{K_s} (\omega_{k,d}) \\
&  \leq  E_{K_s} (\omega_{ m,d}) + 2  \sum_{k=m}^{N} \mathcal{P}_{K_s} (k)\\
\end{align*}
Using the upper bounds from Theorem \ref{thm:Singular Riesz Polarization} one obtains
\begin{align*}
E_{K_s} (\omega_{ N,d}) & \le E_{K_s} (\omega_{ m,d}) + 2  \sum_{k=m}^{N-1} \mathcal{P}_{K_s} (k)  \\
& \leq  E_{K_s} (\omega_{ m,d}) +  2  \sum_{k=m}^{N-1} \big( \mathcal  I_{s,d} k - c_{s,d} k^{\frac{s}{d}}  \big) \\
& =  E_{K_s} (\omega_{ m,d}) +  2  \sum_{k=1}^{N-1} \Big(\mathcal  I_{s,d} k - c_{s,d} k^{\frac{s}{d}} \Big) - 2 \sum_{k=1}^{m-1} \Big(\mathcal  I_{s,d} k - c_{s,d} k^{\frac{s}{d}} \Big)  \\
& = E_{K_s} (\omega_{ m,d}) + \mathcal{I}_{s,d} N (N-1)  - \mathcal{I}_{s,d} m (m-1) - 2 c_{s,d} \Big( \sum_{k=1}^{N-1} k^{\frac{s}{d}} -\sum_{h=1}^{m-1} h^{\frac{s}{d}} \Big).
\end{align*}
We have that for $M \geq 2$
\begin{equation}
\sum_{k=1}^{M-1} k^{\frac{s}{d}} = \begin{cases}
\frac{d}{d+s} M^{1+ \frac{s}{d}} + \mathcal{O}(M^{\frac{s}{d}}), &  s \ge 0,\\ 
\frac{d}{d+s} M^{1+ \frac{s}{d}} + \mathcal{O}(1), &  -d<s<0,\\ 
\log(M) + \mathcal{O}(1), &  s= - d,\\ 
\zeta(-\frac{s}{d}) + \mathcal{O}(M^{1+\frac{s}{d}}), & -2d <s <-d.\\ 
\end{cases}
\end{equation}
and since $N \geq m$, our claim now follows. Note that the case  $s\le -d$  is only relevant when $d = 1$ given that $s \in (-2, d)$.
\end{proof}

\subsection{Proof of Corollary \ref{cor:pol}}\label{sec:cor}
Let $(x_n)_{n \in\mathbb{N}}$ be the  greedy sequence of points on $\mathbb{S}^d$ with respect to $K_s$   as defined in \eqref{eq:greedy}.
Arguing as in the proof of Theorem \ref{thm:greedyenergy},
$$E_{K_s} (\omega_{ N+1,d})  = 2 \sum_{k=1}^{N} \sum_{j=1}^{k} K(x_{k+1}, x_j)  = 2 \sum_{k=1}^{N} P_{K_s} (\omega_{k,d}).$$
We know from Theorem \ref{thm:greedyenergy} that 
$$ E_{K_s}(\omega_{N,d}) \geq \mathcal{I}_{s,d} N^2 - C_{s,d} N^{1 +\frac{s}{d}}.$$
We also recall that
$$P_K(\omega_{N,d}) = \min_{x \in \mathbb{S}^d} \sum_{j=1}^{N} K(x, x_j) \leq \int_{\mathbb{S}^d} \sum_{j=1}^{N} K(x, x_j) d\sigma_d = N \cdot \mathcal{I}_{s,d}.$$
We can now introduce, for $X>0$ and $N \in \mathbb{N}$, the set
$$ A_N = \left\{ 1 \leq j \leq N:  P_{K_s} (\omega_{j,d}) \leq j \cdot \mathcal{I}_{s,d} - X j^{\frac{s}{d}} \right\}.$$
Collecting all these estimates, we see, for some $\alpha_{s,d} > 0$,
\begin{align*}
 \mathcal{I}_{s,d} N^2 - C_{s,d} N^{1 +\frac{s}{d}} &\leq E_{K_s}(\omega_{N,d})  = 2 \sum_{k=1}^{N} P_{K_s} (\omega_{k,d}) \\
 & \le  2 \sum_{k =1 \atop k \notin A_N}^{N} k \cdot \mathcal I_{s,d}  +  2 \sum_{k =1 \atop k \in A_N}^{N}(  k \cdot \mathcal{I}_{s,d} - X k^{\frac{s}{d}} )  \\
  & =  2 \sum_{k =1}^{N} k  \cdot \mathcal{I}_{s,d} - 2   \sum_{ k \in A_N}^{} X k ^{\frac{s}{d}} \\
 & \le  2 \sum_{k =1}^{N} k  \cdot \mathcal{I}_{s,d} - 2  \sum_{k =1}^{\# A_N} X k ^{\frac{s}{d}} \\
 & \le  N^2 \mathcal{I}_{s,d} - \alpha_{s,d} X \left(\# A_N\right)^{1 + \frac{s}{d}}.
\end{align*}
From this we deduce that
$$ \# A_N \leq \left( \frac{C_{s,d}}{\alpha_{s,d} X} \right)^{-\frac{1}{1+\frac{s}{d}}} \cdot N,$$
which is less than $\varepsilon N$ for $X$ large enough.  This proves Corollary \ref{cor:pol}.

\section{$L^2$-discrepancy: Proof of Theorem \ref{thm:L2discrepancy}}
We turn to showing our main result on the $L^2$-spherical cap discrepancy of the greedy sequence. 
The relevance of the greedy construction to discrepancy is based on the following classical  result \cite{stol} that relates the sum of pairwise Euclidean distances (in other words, the Riesz energy with  $s = -1$) to the $L^2$-discrepancy.
\begin{theorem}[Stolarsky Invariance Principle]
\label{thm:stolarsky}
For $\omega_{N,d} = \left\{x_1, \dots, x_N \right\} \subset \mathbb{S}^d$,
\begin{align*}
 (D_{L^2, {\tiny \mbox{cap}}}(\omega_{N,d}))^2 & = c_d \left( \int_{\mathbb{S}^d} \int_{\mathbb{S}^d} \| x-y\| d\sigma_d(x) d\sigma_d(y) - \frac{1}{N^2} \sum_{i,j=1}^{N} \|x_i - x_j\| \right)\\
 & = c_d \, \bigg( \frac{1}{N^2} E_{K_{-1}} (\omega_{N,d})   - \mathcal I_{-1,d} \bigg), 
 \end{align*}
where the constant $c_d$ is given by
$$ c_d = \frac{1}{d} \frac{\Gamma((d+1)/2)}{\sqrt{\pi} \Gamma(d/2)} \sim \frac{1}{\sqrt{2\pi d}} \quad \mbox{as}~d \rightarrow \infty.$$
\end{theorem} 

Note that the construction of the greedy sequence in \eqref{eq:greedy} when $s=-1$ aims to maximize the pairwise sum of Euclidean distances at every step and thus minimize the $L^2$-discrepancy at every step by the Stolarsky invariance principle. As will become clear in our proof of the $N^{-1/2}$ upper bound below, it is not of tremendous importance that one picks the maximum to show such an upper bound -- one really only cares about having a large value in the sum. Obviously, the larger the value the better (and thus the maximum is optimal at that step) but taking values close to the maximum should suffice (and does in practice).

\label{sec:thm1}
\begin{proof}[Proof of Theorem \ref{thm:L2discrepancy}]
Let us assume $\left\{x_1, \dots, x_m \right\} \subset \mathbb{S}^d$ is given and, for all $n \geq m$,  
$$  \sum_{i=1}^{n} \| x_{n+1}- x_i\| \geq n \int_{\mathbb{S}^d} \int_{\mathbb{S}^d} \|x-y\| d\sigma_d(x) d\sigma_d(y).$$
We have the trivial bound
$$ \sum_{i,j=1}^{m} \|x_i - x_j\| \geq 0$$
and, for $n \geq m$, that
\begin{align*}
 \sum_{i,j=1}^{n+1} \|x_i - x_j\| &= \sum_{i,j=1}^{n} \|x_i - x_j\| + 2\sum_{i=1}^{n} \|x_{n+1} - x_i\| \\
 &\geq \sum_{i,j=1}^{n} \|x_i - x_j\| + 2n \int_{\mathbb{S}^d} \int_{\mathbb{S}^d} \|x-y\| d\sigma_d(x) d\sigma_d(y).
 \end{align*}
 Iterating this inequality, we infer, for all $n > m$, that
 $$  \sum_{i,j=1}^{n} \|x_i - x_j\| \geq 2 \left(\sum_{k=m}^{n-1} k \right)\int_{\mathbb{S}^d} \int_{\mathbb{S}^d} \|x-y\| d\sigma_d(x) d\sigma_d(y).$$
 This now implies that for the first $N$ elements
\begin{align*}
(D_{L^2, {\tiny \mbox{cap}}}(\omega_{N,d}))^2 &= c_d \left( \int_{\mathbb{S}^d} \int_{\mathbb{S}^d} \| x-y\| d\sigma_d(x) d\sigma_d(y) - \frac{1}{N^2} \sum_{i,j=1}^{N} \|x_i - x_j\| \right)\\
&\leq c_d \left(\int_{\mathbb{S}^d} \int_{\mathbb{S}^d} \|x-y\| d\sigma_d(x) d\sigma_d(y) \right) \left( 1 - \frac{2}{N^2} \sum_{k=m}^{N-1} k\right)
\end{align*}
from which we deduce that
$$ D_{L^2, {\tiny \mbox{cap}}}(\omega_{N,d}) \leq \sqrt{c_d}  \left(\int_{\mathbb{S}^d} \int_{\mathbb{S}^d} \|x-y\| d\sigma_d(x) d\sigma_d(y) \right)^{\frac{1}{2}}\left( \frac{1}{N} + \frac{m^2-m}{N^2} \right)^{1/2}.$$
 
In the case of $d=2$, we have
$$ c_d = \frac{1}{2} \qquad \mbox{as well as} \qquad -\mathcal{I}_{-1,2} =  \int_{\mathbb{S}^2} \int_{\mathbb{S}^2} \|x-y\| d\sigma_d(x) d\sigma_d(y)  = \frac{4}{3}$$
from which we deduce
$$ D_{L^2, {\tiny \mbox{cap}}}(\omega_{N,d}) \leq  \frac{\sqrt{2}}{\sqrt{3}} \left( \frac{1}{N} + \frac{m^2 - {m} }{N^2} \right)^{1/2}.$$
\end{proof}
Note that the result is sharp (up to lower order terms) for sequences satisfying
$$  \sum_{i=1}^{n} \| x_{n+1}- x_i\| = n \int_{\mathbb{S}^2} \int_{\mathbb{S}^2} \|x-y\| d\sigma_d(x) d\sigma_d(y).$$
It is clear that many such sequences exist: all involved functions are continuous so there are always points where they attain their average value.

\begin{remark}
The above result, in a different form, has more or less appeared in the literature previously: for example, the result follows directly from Theorem 3.1 (and Remark 3.2) in \cite{LopM}. The purpose of reproving it here is to state the result in the form of an $L^2$-discrepancy bound and to give more explicit constants.
\end{remark}

\subsection{Proof of Proposition \ref{prop:disc}}\label{sec:prop:disc}  We conclude with a proof of Proposition \ref{prop:disc}. 
To establish a nontrivial lower bound, we show that the average growth of any consecutive pairs of points is bounded from below.
Suppose $\left\{x_1, \dots, x_n\right\} \subset \mathbb{S}^2$. 
 Recall that $-\mathcal{I}_{-1,2} = \frac{4}{3}$. 
Fix $X \ge 0$ such that 
$$   \max_{x \in \mathbb{S}^2} \sum_{i=1}^{n} \| x - x_i\| = \frac{4n}{3} + X.$$
Assume that $X < \frac{4n}{3}$. In this case it is easy to see that 
$$ \sigma_d \left( \left\{x \in \mathbb{S}^2:  \sum_{i=1}^{n} \| x - x_i\| <  \frac{4n}{3} - X  \right\}\right) < \frac{1}{2}.$$
Indeed,  the average value of the function $ \sum_{i=1}^{n} \| x - x_i\| $  is $4 n/3$
and its maximum is $ 4n/3  + X$, therefore, denoting the set above by $\Omega$, we have 
$$ \frac{4n}{3} < \sigma_d(\Omega) \left( \frac{4n}{3} - X \right) + \big( 1 -\sigma_d(\Omega) \big) \left(\frac{4n}{3}+X\right),$$
which implies that $\sigma_d (\Omega) < 1/2$. 
 Thus every hemisphere contains points $x$ satisfying the inequality $\sum_{i=1}^{n} \| x - x_i\| \ge 4n/3 - X$. 

By restricting  $x$ to the  hemisphere centered at $-x_{n+1}$ we conclude that
\begin{align*} 
\max_{x \in \mathbb{S}^2} \sum_{i=1}^{n+1} \| x - x_i\|  &  =  \max_{x \in \mathbb{S}^2}  \big( \|x - x_{n+1} \| +  \sum_{i=1}^{n} \| x - x_i\|  \big) \\
&  \geq \sqrt{2} + \frac{4n}{3} - X =  \frac{4(n+1)}{3} + \sqrt{2} - \frac{4}{3} - X.
 \end{align*}
Combining this with the definition of $X$, we see that
\begin{equation}\label{eq:sumof2}
 \frac{1}{2}\left(   \max_{x \in \mathbb{S}^2} \sum_{i=1}^{n} \| x - x_i\|  - \frac{4n}{3} + \max_{x \in \mathbb{S}^2}  \sum_{i=1}^{n+1} \| x - x_i\| - \frac{4(n+1)}{3} \right) \geq \frac{1}{\sqrt{2}} - \frac{2}{3}.
 \end{equation}
If $X \ge 4n/3$, one can see immediately that \eqref{eq:sumof2} holds with the right-hand side of $4n/3 - 1/6 >1$. Therefore, the sum $\sum_{n=1}^{N-1} \left( \max_{x \in \mathbb{S}^2}  \sum_{i=1}^{n} \| x - x_i\| -  \frac{4 n}{3} \right)$ in \eqref{eq:averagepol} grows at least linearly in $N$, which leads  to the lower bound in Proposition \ref{prop:disc}.
The upper bound follows from the Stolarsky principle which implies
$$ \int_{\mathbb{S}^d} \int_{\mathbb{S}^d} \| x-y\| d\sigma_d(x) d\sigma_d(y) - \frac{1}{N^2} \sum_{i,j=1}^{N} \|x_i - x_j\| \geq 0$$
and thus, for any set of points $\left\{x_1, \dots, x_N \right\} \subset \mathbb{S}^2$,
$$  \sum_{i,j=1}^{N} \|x_i - x_j\| \leq \frac{4N^2}{3}.$$
For a greedy sequence, we can write
\begin{align*}
 \sum_{i,j=1}^{N} \|x_i - x_j\| &= 2 \sum_{n=1}^{N-1} \left[ \frac{4n}{3} + \left(\max_{x \in \mathbb{S}^2} \sum_{i=1}^{n} \| x - x_i\|  - \frac{4n}{3} \right) \right] \\
 &= \frac{4}{3} (N-1)N+ 2 \sum_{n=1}^{N-1} \left(\max_{x \in \mathbb{S}^2} \sum_{i=1}^{n} \| x - x_i\|  - \frac{4n}{3} \right)
\end{align*}
and thus
$$ \frac{1}{N}\sum_{n=1}^{N-1} \left(\max_{x \in \mathbb{S}^2} \sum_{i=1}^{n} \| x - x_i\|  - \frac{4n}{3} \right) \leq \frac{2}{3}.$$

Note that this upper bound (with a worse constant) could be obtained directly from the lower bound in \eqref{eq:opt_polarization1} of Theorem \ref{thm:Singular Riesz Polarization}. We also remark that the same arguments  would also work in higher dimensions (with different constants).

\section{Numerical Examples and Comments}
\label{sec:numerics}
This section contains some basic numerical examples and general comments. First we note that, in the greedy construction, finding the exact point
 $$x_{n+1} = \arg\max_{x \in \mathbb{S}^d} \sum_{i=1}^{n} \| x- x_i\|$$
is computationally nontrivial. 
\vspace{-10pt}
\begin{center}
\begin{figure}[h!]
\begin{tikzpicture}
\node at (0,0) {\includegraphics[width=0.4\textwidth]{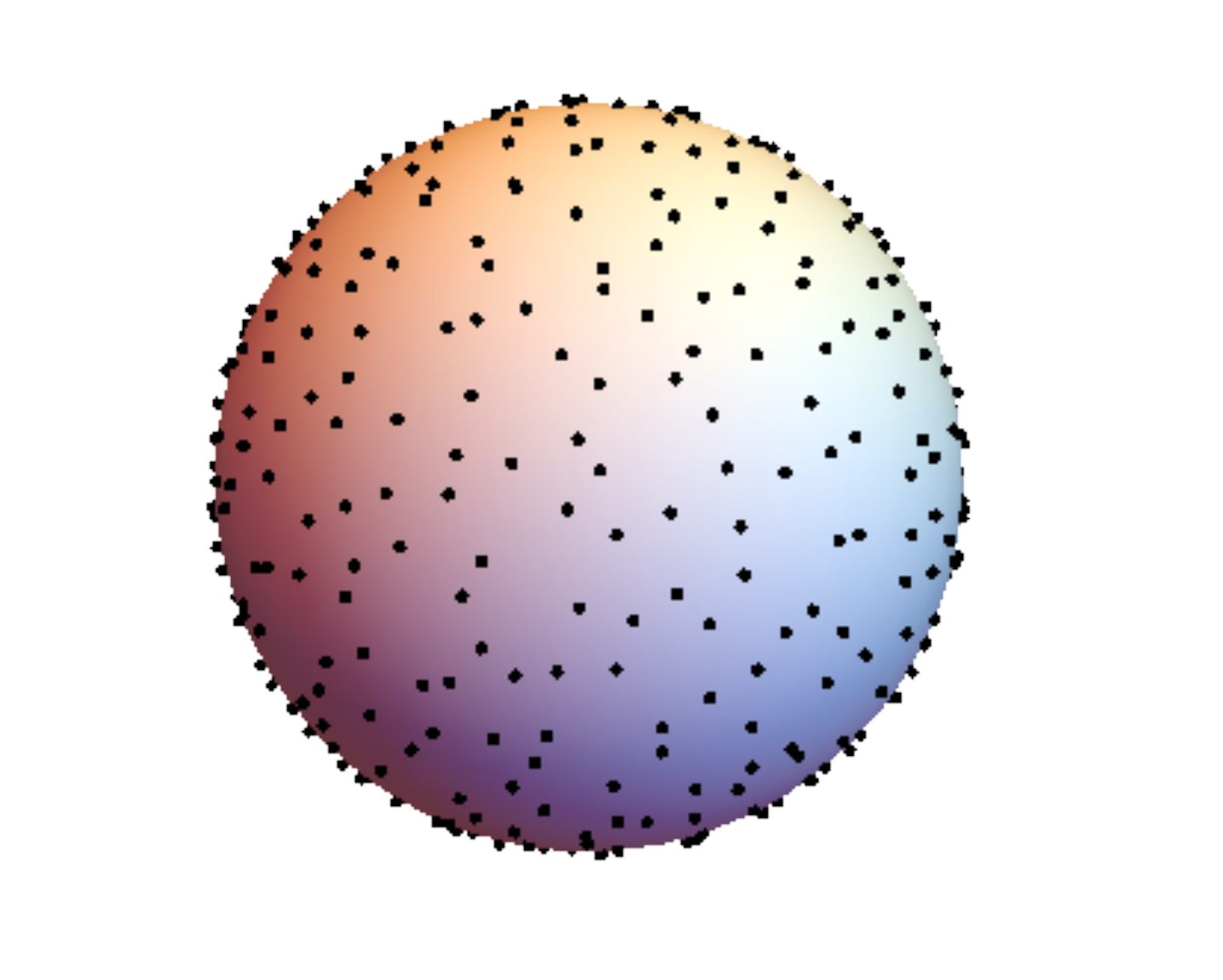}};
\node at (6,0) {\includegraphics[width=0.4\textwidth]{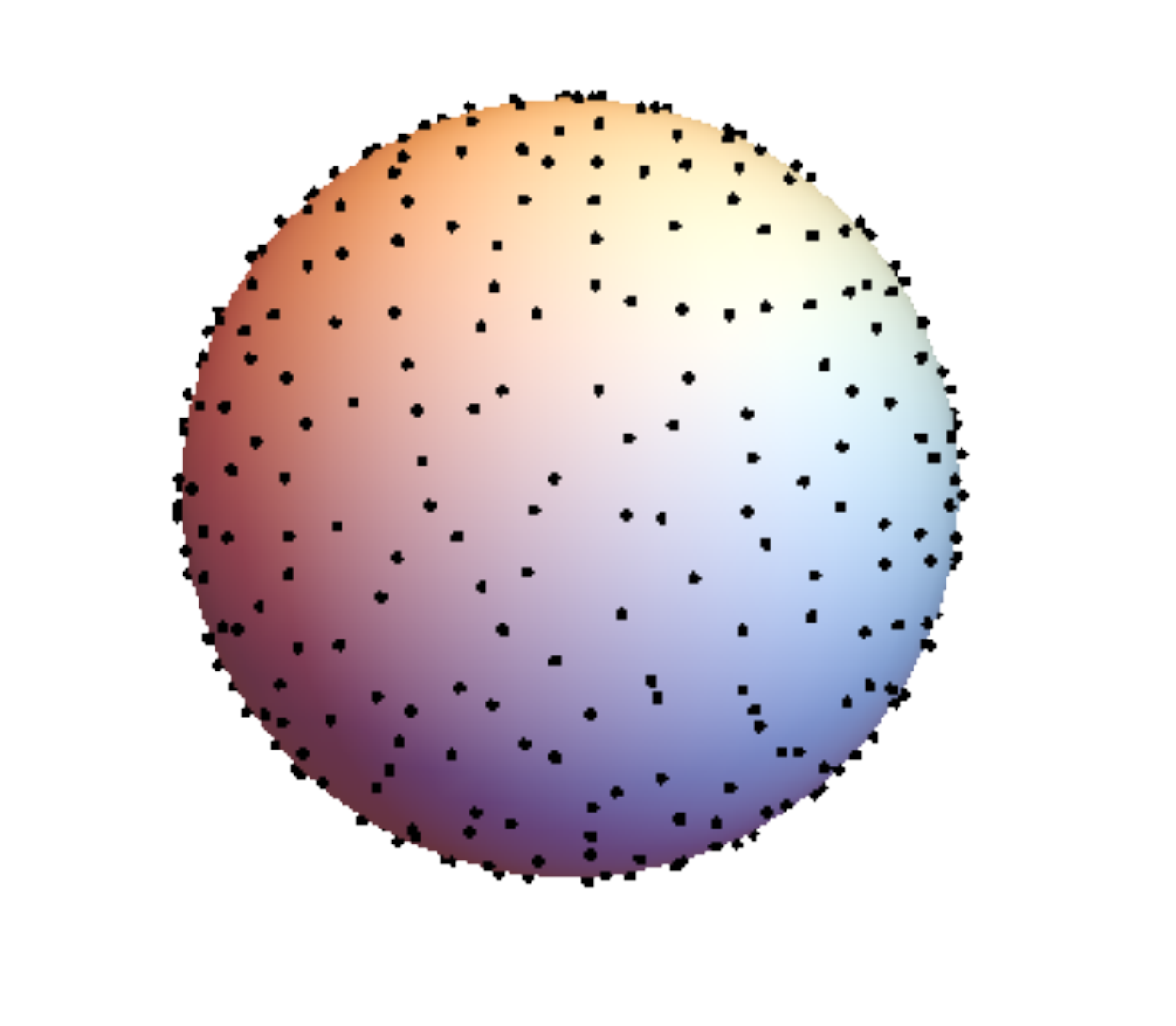}};
\end{tikzpicture}
\vspace{-20pt}
\caption{An example of a greedy point set with $N=500$ points on $\mathbb{S}^2$ seen from two angles (started with a single point).}
\label{fig:points}
\end{figure}
\end{center}
If the success of the greedy construction were to depend very strongly on finding exact maxima in each step, it would not be a very useful construction to begin with (indeed, as also indicated by the proof of Theorem \ref{thm:L2discrepancy}, luckily this does not seem to be the case). Numerical experiments suggest the exact opposite: these types of sequences tend to be incredibly robust and seem to lead to good results even if sometimes adversarial points are added manually. Throughout this section, we consider approximate sequences obtained as follows: given $\left\{x_1, \dots, x_N\right\} \subset \mathbb{S}^d$, we consider
100 random points and add the one maximizing the sum of distances among those. 
Figure \ref{fig:points} shows an example of a set of $N=500$ points on $\mathbb{S}^2$ obtained from a single individual point. We observe that the sequence looks somewhat random but avoids clusters of points more so than an actually random sequence would.

\begin{center}
\begin{figure}[h!]
\begin{tikzpicture}
\node at (0,0) {\includegraphics[width=0.45\textwidth]{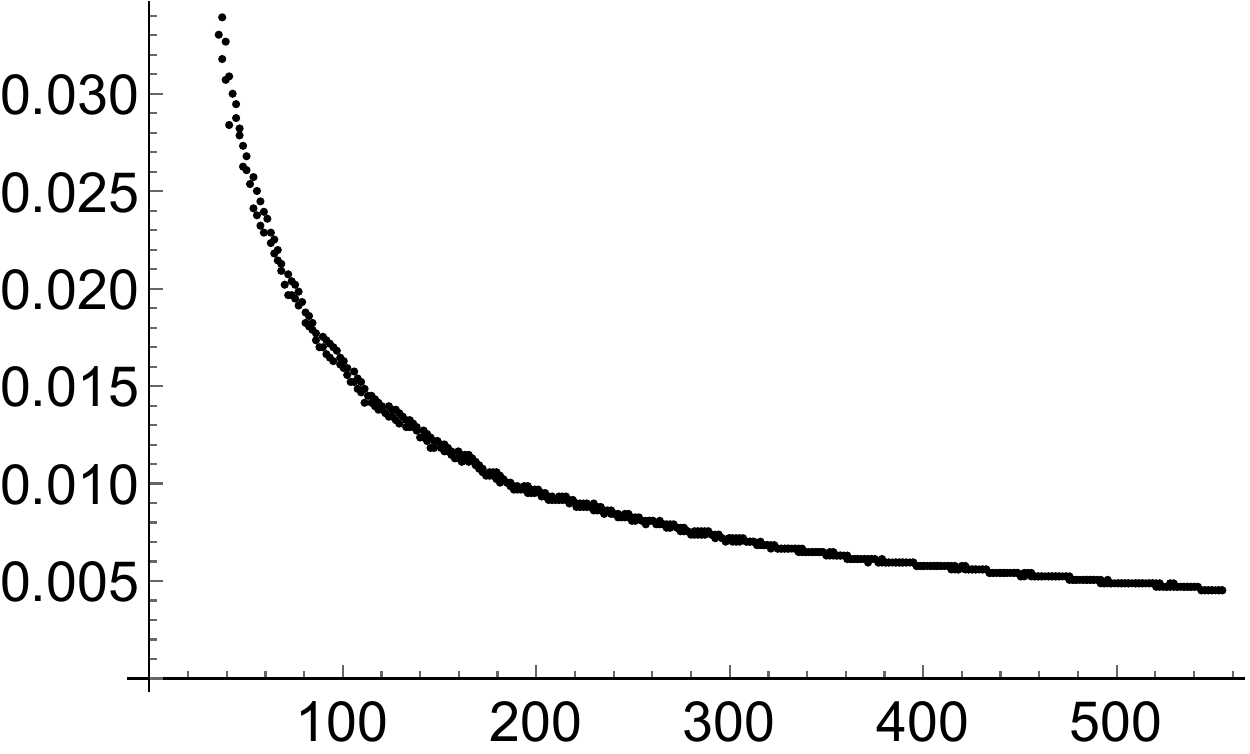}};
\node at (6.5,0) {\includegraphics[width=0.45\textwidth]{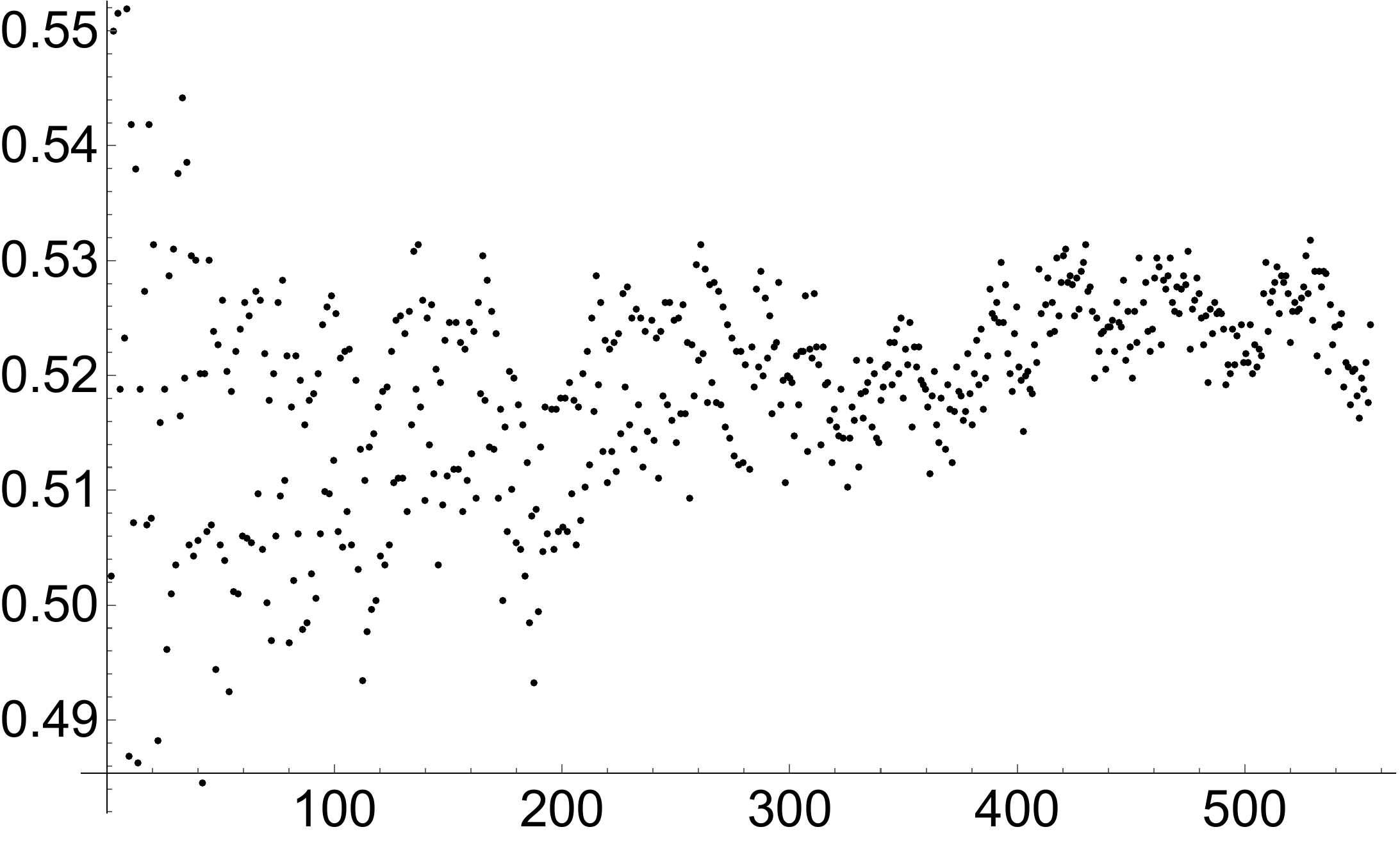}};
\end{tikzpicture}
\vspace{-10pt}
\caption{$L^2-$discrepancy of a greedy sequence on $\mathbb{S}^2$ started with one point. Left: $N^{1/2} \cdot D_{L^2, \tiny \mbox{cap}}(\omega_{N})$, right: $N^{3/4} \cdot D_{L^2, \tiny \mbox{cap}}(\omega_{N})$.}
\label{fig:S2}
\end{figure}
\end{center}
\vspace{-10pt}

When plotting the $L^2-$spherical cap discrepancy of this sequence one observes, numerically, that the discrepancy seems to be relatively close to $\sim 0.5 \cdot N^{-3/4}$. Note that our construction of the sequence was based on the approximation by random points and it stands to reason that Figure \ref{fig:S2} serves as an upper bound of the true behavior of a greedy sequence.
 The greedy method is agnostic to what happened in the past and works well with any arbitrary initial set. We illustrate this with a simple example where we first take 250 points uniformly at random and then compute another 250 points greedily. 
\begin{center}
\begin{figure}[h!]
\begin{tikzpicture}
\node at (0,0) {\includegraphics[width=0.45\textwidth]{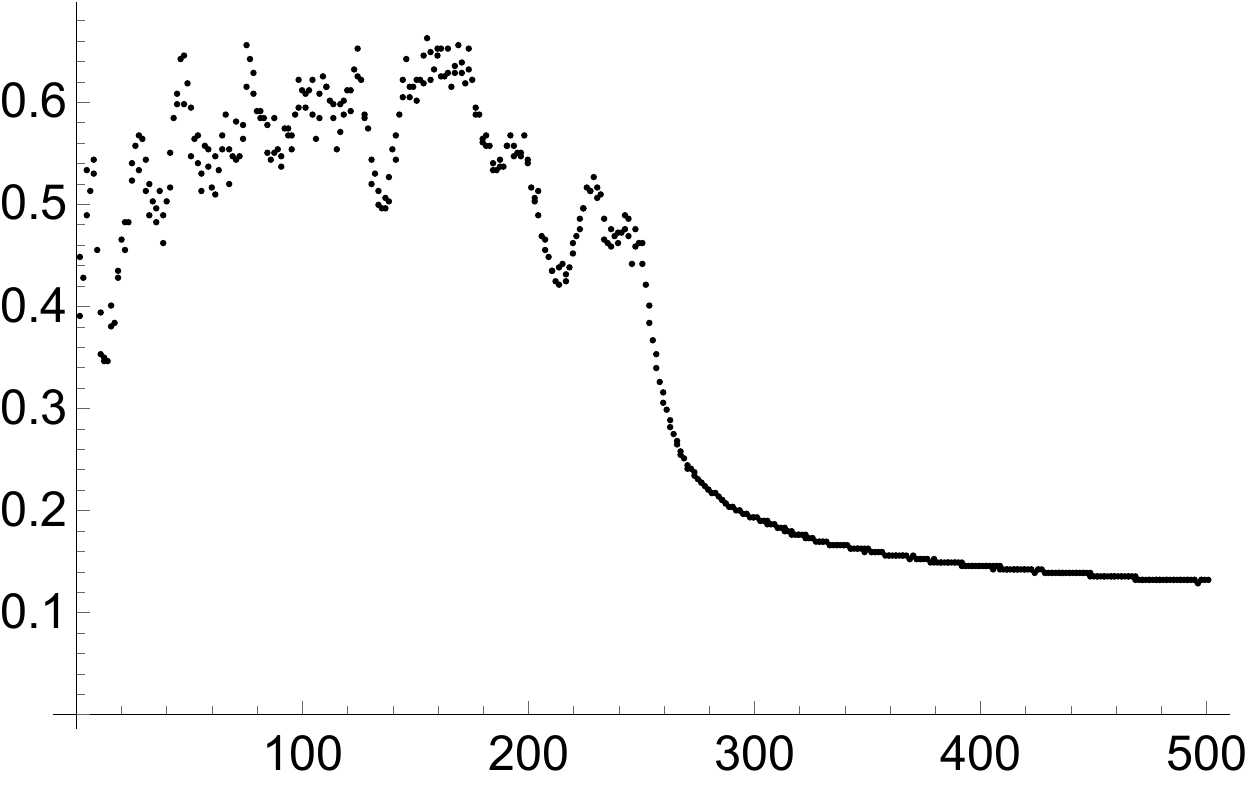}};
\node at (6.5,0) {\includegraphics[width=0.45\textwidth]{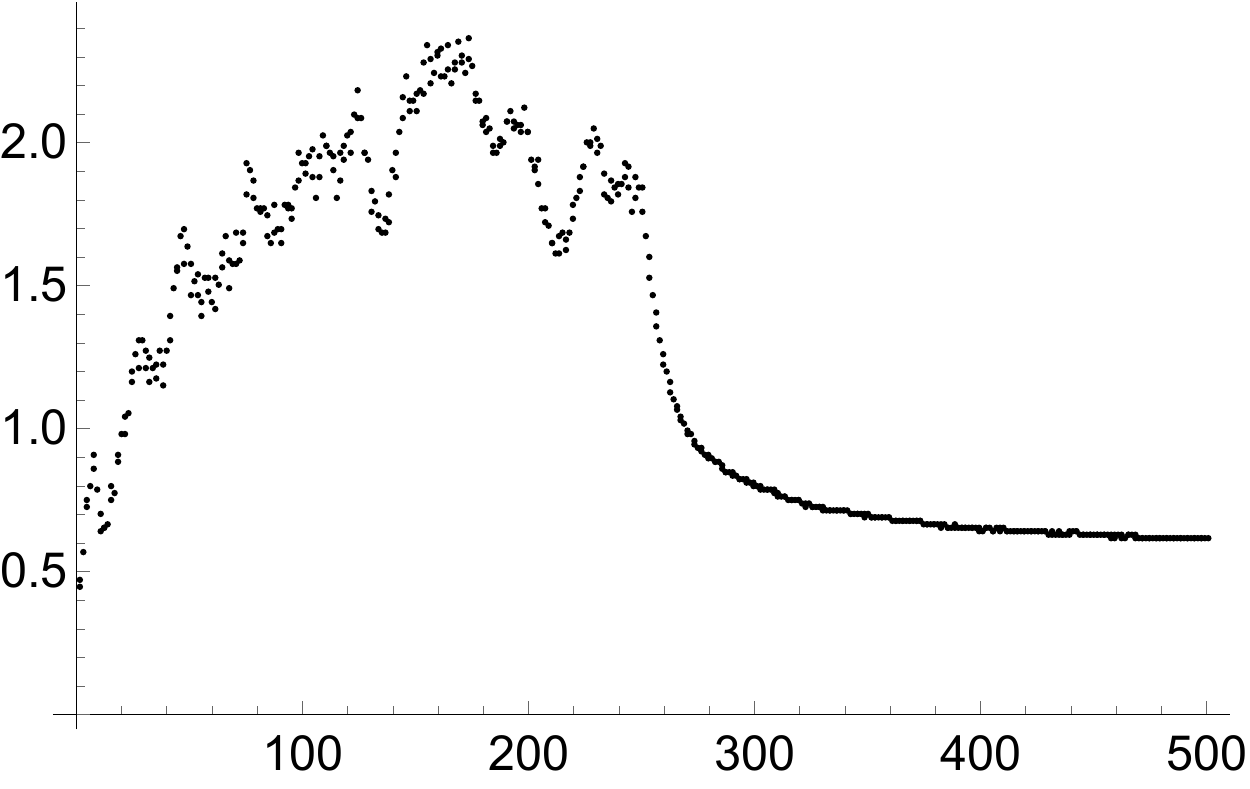}};
\end{tikzpicture}
\vspace{-10pt}
\caption{$L^2-$discrepancy of a greedy sequence comprised of 250 random points and then followed by 250 points using the greedy construction. Left: $N^{1/2} \cdot D_{L^2, \tiny \mbox{cap}}(\omega_{N})$, right: $N^{3/4} \cdot D_{L^2, \tiny \mbox{cap}}(\omega_{N}))$.}
\end{figure}
\end{center}
\vspace{-10pt}
The result is striking: for the first 250 elements, we see that $N^{1/2}  \cdot D_{L^2, \tiny \mbox{cap}}(\omega_{N})$ is approximately constant (as expected for random points). After that there is a pronounced decay.  Perhaps even more striking is that $N^{3/4}  \cdot D_{L^2, \tiny \mbox{cap}}(\omega_{N})$ is first increasing (roughly at rate $\sim N^{1/4}$ as we expect) and then quickly decreases and returns to a constant slightly above $1/2$ (see also Fig. 2).

\begin{center}
\begin{figure}[h!]
\begin{tikzpicture}
\node at (0,0) {\includegraphics[width=0.45\textwidth]{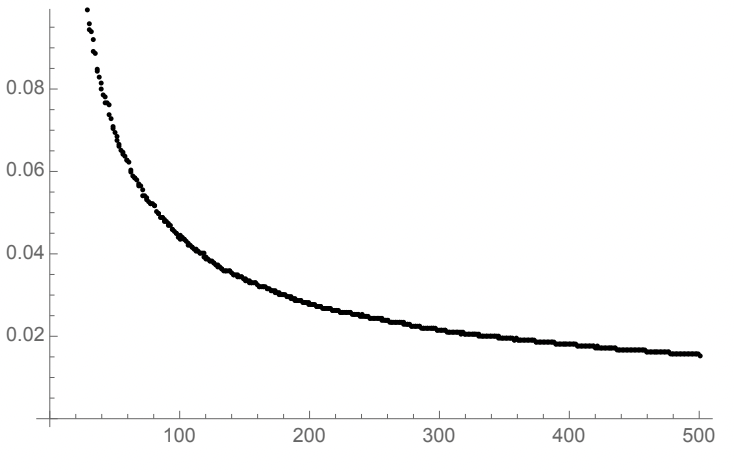}};
\node at (6.5,0) {\includegraphics[width=0.45\textwidth]{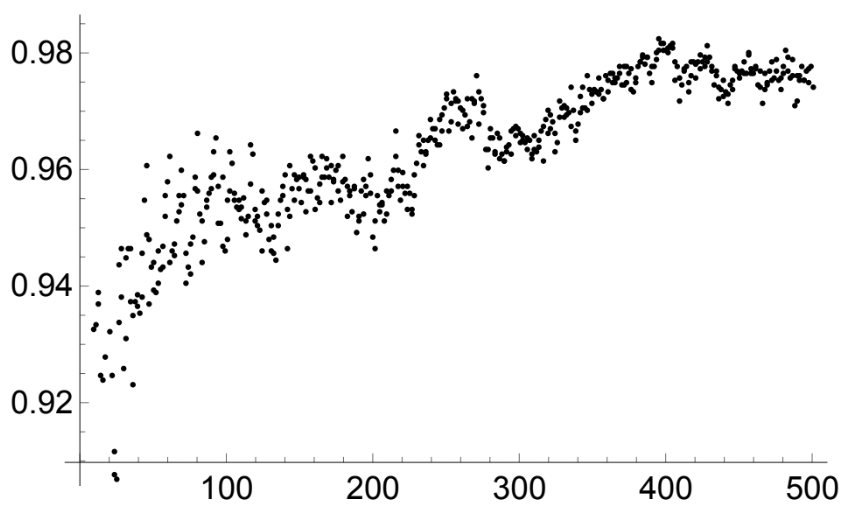}};
\end{tikzpicture}
\vspace{-10pt}
\caption{$L^2-$discrepancy of a sequence on $\mathbb{S}^3$ started with a single point. Left: $N^{1/2} \cdot D_{L^2, \tiny \mbox{cap}}(\omega_{N,3})$, right: $N^{2/3} \cdot D_{L^2, \tiny \mbox{cap}}(\omega_{N,3})$.}
\label{fig:s3}
\end{figure}
\end{center}
\vspace{-10pt}

One could wonder about higher dimensions: one would expect optimal sequences to behave as $D_{L^2, \tiny \mbox{cap}}(\omega_{N,3}) \sim N^{-2/3}$ on $\mathbb{S}^3$ as well as
$D_{L^2, \tiny \mbox{cap}}(\omega_{N,4}) \sim N^{ -5/8}$ on $\mathbb{S}^4$.
The same basic numerical experiment leads to results compatible with this interpretation. We emphasize that, due to increasing computational cost, these experiments were carried out only for rather small values of $N \leq 500$. At this scale, small powers of $N$ or logarithms are not always easy to detect. Moreover, the effectiveness our numerical procedure of picking 100 points at random and then adding the one with the largest distance depends on the profile of the function: if large deviations occur on sets of small measure (an effective that could conceivably become more pronounced in higher dimensions), then this method will lose effectiveness. Regardless, we believe these preliminary results to be rather interesting and hope they will inspire subsequent work on the regularity of greedy sequences on $\mathbb{S}^d$.

\begin{center}
\begin{figure}[h!]
\begin{tikzpicture}
\node at (0,0) {\includegraphics[width=0.45\textwidth]{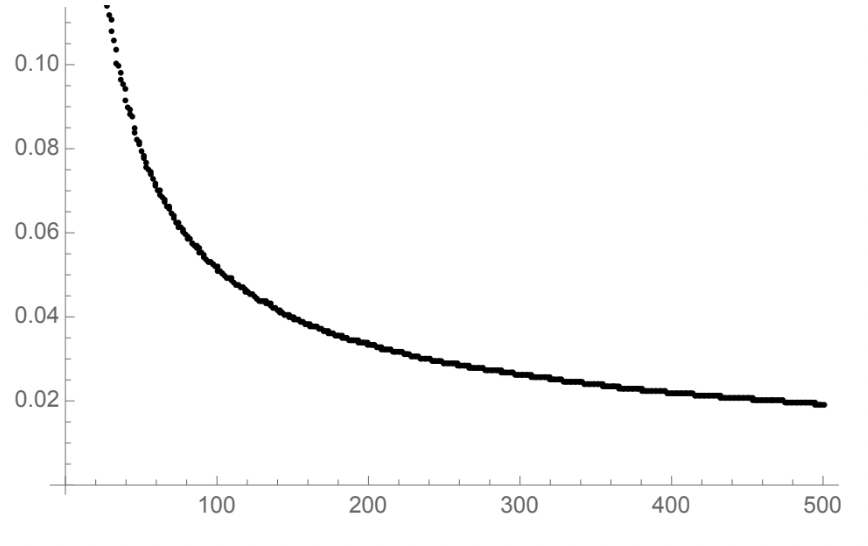}};
\node at (7,0) {\includegraphics[width=0.45\textwidth]{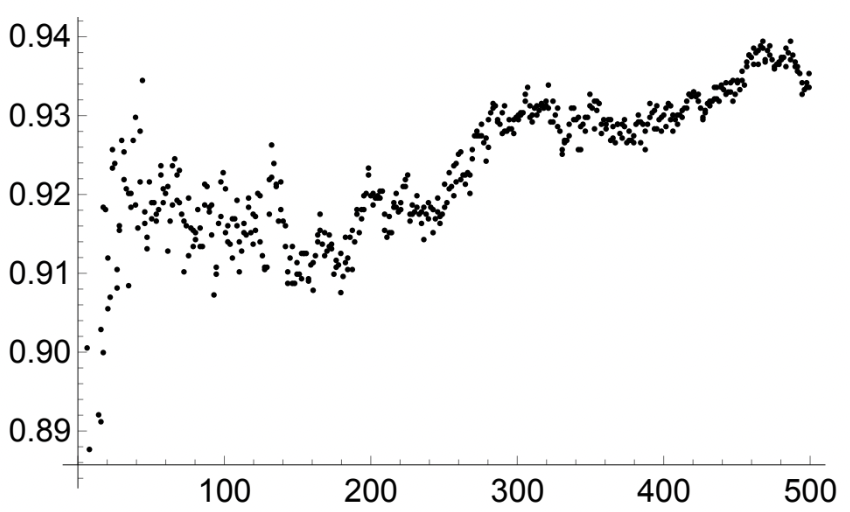}};
\end{tikzpicture}
\vspace{-10pt}
\caption{$L^2-$discrepancy of a sequence on $\mathbb{S}^4$ started with a single point. Left: $N^{1/2} \cdot D_{L^2, \tiny \mbox{cap}}(\omega_{N,4})$, right: $N^{5/8} \cdot D_{L^2, \tiny \mbox{cap}}(\omega_{N,4})$.}
\label{fig:S4}
\end{figure}
\end{center}
\vspace{-10pt}

We conclude with a conjecture which collects our observations in this section and speculates that the greedy sequence has nearly  optimal spherical cap discrepancy.

\begin{conjecture}
Let $\omega_{N,d} \subseteq \mathbb{S}^d$ be the first $N$ elements of the greedy sequence with respect to $K_{-1}$  as defined  in \eqref{eq:greedy}.  Then, for some $c = c(d)  \ge 0$, 
$$D_{L^2, \tiny \mbox{cap}}(\omega_{N,d}) = \mathcal O \big(  N^{-\frac{1}{2}-\frac{1}{2d}} \log^c(N) \big).$$
\end{conjecture}

\section{Appendix}\label{sec:Appendix}

\subsection{Gegenbauer coefficients}

In this section we compute the Gegenbauer coefficients of the functions $f_s$ (the Riesz kernel) and $f_{s,\varepsilon}$ (its approximation) which were defined in \S\ref{sec:upper}. Various computations for such coefficients for some ranges of $s$ can be found scattered in the literature (e.g. \cite{Brauchart,DG}). However, since we need finer properties of these coefficients (asymptotics, monotonicity, positivity) in the full range, we present detailed computations here. 
We proceed first with finding the Gegenbauer coefficients for the functions $f_{s, \varepsilon}$.

\begin{lemma}\label{lem:Riesz epsilon Gegen Coeff}
For $0< s $, $n \in \mathbb{N}$
\begin{align*}
\widehat{f_{s, \varepsilon}}(n,d) &= \frac{\Gamma( \frac{d+1}{2}) \Gamma(\frac{s}{2} +n)}{\Gamma( n + \frac{d+1}{2}) \Gamma( \frac{s}{2})} (2+ \varepsilon)^{-n - \frac{s}{2}} \Hypergeom21%
{ \frac{2n+s}{4}, \frac{2n+2+s}{4}}%
{n + \frac{d+1}{2}}{\bigg( \frac{2}{2+\varepsilon} \bigg)^2} \\
& = \frac{\Gamma( \frac{d+1}{2})}{\Gamma(\frac{s}{2})} (2+\varepsilon)^{-n - \frac{s}{2}} \sum_{k=0}^{\infty} \frac{\Gamma( n + 2k + \frac{s}{2})}{k! \Gamma( n + k + \frac{d+1}{2})} (2 + \varepsilon)^{-2k},
\end{align*}
where $ \sideset{_2}{_1}\HyperF$ is the ordinary hypergeometric function.
\end{lemma}

\begin{proof}
On multiple occasions, we will use the identity 
\begin{equation}
\sqrt{\pi} \Gamma(2a) = 2^{2a-1} \Gamma(a) \Gamma \bigg( a + \frac{1}{2} \bigg), \quad a > 0.
\end{equation}
We shall also use the Rodrigues formula
\begin{equation}\label{eq:Rodrigues}
C_n^{\frac{d-1}{2}}(t) = \frac{(-2)^n \Gamma(n + \frac{d-1}{2}) \Gamma(n+d-1)}{n! \Gamma(\frac{d-1}{2}) \Gamma(2n+ d-1)} (1-t^2)^{\frac{2-d}{2}} \bigg( \frac{d}{dt} \bigg)^n ( 1-t^2)^{n+ \frac{d-2}{2}}.
\end{equation}

Using \eqref{eq:Rodrigues} with \eqref{eq:GegenCoeff}, we have, for $0 < s < d$ and $\varepsilon > 0$

\begin{align*}
\widehat{f_{s, \varepsilon}}(n,d) =& \hspace{0.1cm} 2^{-\frac{s}{2}} \frac{\Gamma(\frac{d+1}{2}) n! \Gamma(d-1)}{\sqrt{\pi} \Gamma( \frac{d}{2}) \Gamma(n+d-1)} \frac{(-2)^n \Gamma(n + \frac{d-1}{2}) \Gamma(n+d-1)}{n! \Gamma(\frac{d-1}{2}) \Gamma(2n+ d-1)}\\ &\times \int_{-1}^{1} (1 + \frac{\varepsilon}{2}-t)^{- \frac{s}{2}} \big( \frac{d}{dt} \big)^n ( 1-t^2)^{n+ \frac{d-2}{2}} dt \\
= & \hspace{0.1cm} 2^{n-\frac{s}{2}} \frac{\Gamma(\frac{d+1}{2}) \Gamma(d-1)}{\sqrt{\pi} \Gamma( \frac{d}{2}) } \frac{ \Gamma(n + \frac{d-1}{2}) \Gamma(\frac{s}{2}+n)}{ \Gamma(\frac{d-1}{2}) \Gamma(2n+ d-1) \Gamma(\frac{s}{2})} \\ & \times \int_{-1}^{1} (1 + \frac{\varepsilon}{2}-t)^{- \frac{s}{2}-n } ( 1-t^2)^{n+ \frac{d-2}{2}} dt \\
= & \hspace{0.1cm} 2^{n-\frac{s}{2}} \frac{\Gamma(\frac{d+1}{2}) \Gamma(d-1)}{\sqrt{\pi} \Gamma( \frac{d}{2}) } \frac{ \Gamma(n + \frac{d-1}{2}) \Gamma(\frac{s}{2}+n)}{ \Gamma(\frac{d-1}{2}) \Gamma(2n+ d-1) \Gamma(\frac{s}{2})} \Big( \frac{2}{2+\varepsilon} \Big)^{n + \frac{s}{2}} \sqrt{\pi} \\
& \quad \times \frac{\Gamma( n + \frac{d}{2})}{\Gamma( n + \frac{d+1}{2})} \Hypergeom21%
{ \frac{2n+s}{4}, \frac{2n+2+s}{4}}%
{n + \frac{d+1}{2}}{\bigg( \frac{2}{2+\varepsilon} \bigg)^2} \\
= & \hspace{0.1cm} \frac{\Gamma( \frac{d+1}{2}) \Gamma(\frac{s}{2} +n)}{\Gamma( n + \frac{d+1}{2}) \Gamma( \frac{s}{2})} (2+ \varepsilon)^{-n - \frac{s}{2}} \Hypergeom21%
{ \frac{2n+s}{4}, \frac{2n+2+s}{4}}%
{n + \frac{d+1}{2}}{\bigg( \frac{2}{2+\varepsilon} \bigg)^2},
\end{align*}
giving us the first expression in the claim. For the second one,
\begin{align*}
\widehat{f_{s, \varepsilon}}(n,d) = & \hspace{0.1cm} \frac{\Gamma( \frac{d+1}{2}) \Gamma(\frac{s}{2} +n)}{\Gamma( n + \frac{d+1}{2}) \Gamma( \frac{s}{2})} (2+ \varepsilon)^{-n - \frac{s}{2}} \Hypergeom21%
{ \frac{2n+s}{4}, \frac{2n+2+s}{4}}%
{n + \frac{d+1}{2}}{\bigg( \frac{2}{2+\varepsilon} \bigg)^2}\\
 = & \hspace{0.1cm} \frac{\Gamma( \frac{d+1}{2}) \Gamma(\frac{s}{2} +n)}{\Gamma( n + \frac{d+1}{2}) \Gamma( \frac{s}{2})} (2+ \varepsilon)^{-n - \frac{s}{2}} \\ & \times \sum_{k=0}^{\infty} \frac{\Gamma( \frac{n}{2} + \frac{s}{4} + k) \Gamma( \frac{n+1}{2} + \frac{s}{4} + k) \Gamma(n + \frac{d+1}{2})}{k! \Gamma( \frac{n}{2} + \frac{s}{4}) \Gamma( \frac{n+1}{2} + \frac{s}{4}) \Gamma(n + \frac{d+1}{2} + k)} \bigg( \frac{2}{2+\varepsilon} \bigg)^{2k}\\
 = & \hspace{0.1cm}  2^{n + \frac{s}{2}-1} \frac{\Gamma( \frac{d+1}{2})}{\Gamma( \frac{s}{2}) \sqrt{\pi}} (2+ \varepsilon)^{-n - \frac{s}{2}} \sum_{k=0}^{\infty} \frac{\Gamma( \frac{n}{2} + \frac{s}{4} + k) \Gamma( \frac{n+1}{2} + \frac{s}{4} + k)}{k! \Gamma(n + \frac{d+1}{2} + k)} \bigg( \frac{2}{2+\varepsilon} \bigg)^{2k}\\
= & \hspace{0.1cm} \frac{\Gamma( \frac{d+1}{2})}{\Gamma(\frac{s}{2})} (2+\varepsilon)^{-n - \frac{s}{2}} \sum_{k=0}^{\infty} \frac{\Gamma( n + 2k + \frac{s}{2})}{k! \Gamma( n + k + \frac{d+1}{2})} (2 + \varepsilon)^{-2k}.
\end{align*}
\end{proof}

\begin{corollary}\label{cor:Riesz epsilon Gegen Coeff}
For $0< s $, $n \in \mathbb{N}_0$, the Gegenbauer coefficients $\widehat{f_{s, \varepsilon}}(n,d)$ are positive and decreasing in $n$.
\end{corollary}
\begin{proof}
For $0 < s < d$ and $n \in \mathbb{N}_0$, we see that all the summands of
\begin{equation}\label{eq:Gegen Epsilon summands}
\sum_{k=0}^{\infty} \frac{\Gamma( n + 2k + \frac{s}{2})}{k! \Gamma( n + k + \frac{d+1}{2})} (2 + \varepsilon)^{-2k-n-\frac{s}{2}}
\end{equation}
are positive, so $\widehat{f_{s, \varepsilon}}(n,d) > 0$. We also see that for $k \in \mathbb{N}_0$,
\begin{equation*}
\frac{\frac{\Gamma( n + 2k + \frac{s}{2})}{ k! \Gamma( n + k + \frac{d+1}{2})} (2+\varepsilon)^{-n - \frac{s}{2}-2k}}{\frac{\Gamma( (n+1) + 2k + \frac{s}{2})}{ k! \Gamma( (n+1) + k + \frac{d+1}{2})} (2+\varepsilon)^{-(n+1) - \frac{s}{2}-2k}} =  \frac{n+k + \frac{d+1}{2}}{n+2k+\frac{s}{2}} (2+\varepsilon) > 1
\end{equation*}
so the summands in \eqref{eq:Gegen Epsilon summands}
are decreasing as a function of $n$. Thus, $ \widehat{f_{s, \varepsilon}}(n,d)$ is indeed a decreasing function in $n$, for $0 < s < d$ and any $\varepsilon > 0$.
\end{proof}

We know consider the Gegenbauer coefficients of the Riesz kernels themselves.
\begin{lemma}\label{lem:Riesz Gegen Coeff}
For $-2< s < d$, $n \in \mathbb{N}$
\begin{equation}
\widehat{f_s}(n,d) = \begin{cases}
\sgn(s) 2^{d-s-1} \frac{\Gamma(\frac{d+1}{2}) \Gamma( \frac{d-s}{2})  }{\sqrt{\pi} \Gamma(\frac{s}{2})} \frac{\Gamma(n+ \frac{s}{2})}{ \Gamma( n+d-\frac{s}{2})}  & s \neq 0 \\
2^{d-2} \frac{\Gamma(\frac{d+1}{2}) \Gamma( \frac{d}{2}) }{\sqrt{\pi} } \frac{\Gamma(n)}{ \Gamma( n+d)} & s = 0
\end{cases}.
\end{equation}
\end{lemma}

\begin{proof}
We again use the Rodrigues formula along with integration by parts to find, for $s \in (-2,0) \cup (0,d)$, and $n \in \mathbb{N}_0$,
\begin{align*}
\widehat{f_s}(n,d) & = \sgn(s) 2^{-\frac{s}{2}} \frac{\Gamma(\frac{d+1}{2}) n! \Gamma(d-1)}{\sqrt{\pi} \Gamma( \frac{d}{2}) \Gamma(n+d-1)} \frac{(-2)^n \Gamma(n + \frac{d-1}{2}) \Gamma(n+d-1)}{n! \Gamma(\frac{d-1}{2}) \Gamma(2n+ d-1)} \\
& \quad \times \int_{-1}^{1} (1-t)^{- \frac{s}{2}} \big( \frac{d}{dt} \big)^n ( 1-t^2)^{n+ \frac{d-2}{2}} dt \\
& = \sgn(s) 2^{n-\frac{s}{2}} \frac{\Gamma(\frac{d+1}{2}) \Gamma(d-1)}{\sqrt{\pi} \Gamma( \frac{d}{2}) } \frac{ \Gamma(n + \frac{d-1}{2}) \Gamma(\frac{s}{2}+n)}{ \Gamma(\frac{d-1}{2}) \Gamma(2n+ d-1) \Gamma(\frac{s}{2})}\\
& \quad \times \int_{-1}^{1} (1-t)^{- \frac{s}{2}-n } ( 1-t^2)^{n+ \frac{d-2}{2}} dt \\
& = \sgn(s) 2^{n-\frac{s}{2}} \frac{\Gamma(\frac{d+1}{2}) \Gamma(d-1)}{\sqrt{\pi} \Gamma( \frac{d}{2}) } \frac{ \Gamma(n + \frac{d-1}{2}) \Gamma(\frac{s}{2}+n)}{ \Gamma(\frac{d-1}{2}) \Gamma(2n+ d-1) \Gamma(\frac{s}{2})}\\
& \quad \times \int_{-1}^{1} (1-t)^{\frac{d-s-2}{2} } ( 1+t)^{n+ \frac{d-2}{2}} dt \\
& = \sgn(s) 2^{n-\frac{s}{2}} \frac{\Gamma(\frac{d+1}{2}) \Gamma(d-1)}{\sqrt{\pi} \Gamma( \frac{d}{2}) } \frac{ \Gamma(n + \frac{d-1}{2}) \Gamma(\frac{s}{2}+n)}{ \Gamma(\frac{d-1}{2}) \Gamma(2n+ d-1) \Gamma(\frac{s}{2})} \\
& \quad \times\frac{2^{n+d-\frac{s}{2}-1} \Gamma( \frac{d-s}{2}) \Gamma( n + \frac{d}{2})}{(n+d-\frac{s}{2}-1) \Gamma( n+d-\frac{s}{2}-1)}\\
& = \sgn(s) 2^{d-s-1} \frac{\Gamma(\frac{d+1}{2}) }{\sqrt{\pi} } \frac{  \Gamma(\frac{s}{2}+n)}{ \Gamma(\frac{s}{2})} \frac{\Gamma( \frac{d-s}{2}) }{ \Gamma( n+d-\frac{s}{2})}.
\end{align*}
For $s=0$ and $n \geq 1$, we have a similar result:
\begin{align*}
\widehat{f_s}(n,d) & = - \frac{1}{2} \frac{\Gamma(\frac{d+1}{2}) n! \Gamma(d-1)}{\sqrt{\pi} \Gamma( \frac{d}{2}) \Gamma(n+d-1)} \frac{(-2)^n \Gamma(n + \frac{d-1}{2}) \Gamma(n+d-1)}{n! \Gamma(\frac{d-1}{2}) \Gamma(2n+ d-1)} \\
& \quad \times\int_{-1}^{1} \log(1-t) \big( \frac{d}{dt} \big)^n ( 1-t^2)^{n+ \frac{d-2}{2}} dt \\
& = 2^{n-1} \frac{\Gamma(\frac{d+1}{2})  \Gamma(d-1)}{\sqrt{\pi} \Gamma( \frac{d}{2}) } \frac{ \Gamma(n + \frac{d-1}{2}) (n-1)! }{ \Gamma(\frac{d-1}{2}) \Gamma(2n+ d-1)} \\
& \quad \times\int_{-1}^{1} (1-t)^{1-n} ( 1-t^2)^{n+ \frac{d-2}{2}} dt \\
& = 2^{n-1} \frac{\Gamma(\frac{d+1}{2})  \Gamma(d-1)}{\sqrt{\pi} \Gamma( \frac{d}{2}) } \frac{ \Gamma(n + \frac{d-1}{2}) (n-1)! }{ \Gamma(\frac{d-1}{2}) \Gamma(2n+ d-1)} \\
& \quad \times\int_{-1}^{1} (1-t)^{\frac{d}{2}} ( 1+t)^{n+ \frac{d-2}{2}} dt \\
& = 2^{n-1} \frac{\Gamma(\frac{d+1}{2}) \Gamma(d-1)}{\sqrt{\pi} \Gamma( \frac{d}{2}) } \frac{ \Gamma(n + \frac{d-1}{2}) (n-1)!}{ \Gamma(\frac{d-1}{2}) \Gamma(2n+ d-1) } \frac{2^{n+d-1} \Gamma( \frac{d}{2}) \Gamma( n + \frac{d}{2})}{\Gamma( n+d)}\\
& = 2^{d-2} \frac{\Gamma(\frac{d+1}{2}) \Gamma( \frac{d}{2}) }{\sqrt{\pi} } \frac{\Gamma(n)}{ \Gamma( n+d)}.
\end{align*}
\end{proof}

The coefficients for $n \in \mathbb{N}$ are thus clearly positive, and a comparison of consecutive coefficients in all cases also shows that $\widehat{f_s}(n,d)$ is also decreasing as a function of $n$. A quick asymptotic analysis give us the following:
\begin{corollary}\label{cor:Riesz Gegen Coeff}
For $-2< s < d$, $n \in \mathbb{N}$, $\widehat{f_s}(n,d)$ is positive and decreasing as a function of $n$. Moreover
\begin{equation}\label{eq:Riesz Gegen Coeff Asymptotics}
\widehat{f_s}(n,d) = \begin{cases}
\sgn(s) 2^{d-s-1} \frac{\Gamma(\frac{d+1}{2}) \Gamma( \frac{d-s}{2})}{\sqrt{\pi} \Gamma(\frac{s}{2})} n^{s-d} + \mathrm{O}(n^{s-d-1}) & s \neq 0 \\
2^{d-2} \frac{\Gamma(\frac{d+1}{2}) \Gamma( \frac{d}{2}) }{\sqrt{\pi} } n^{-d} + \mathrm{O}(n^{-d-1}) & s = 0
\end{cases}.
\end{equation}
\end{corollary}

We note that in the case that $d=1$, we actually have the Chebyshev polynomials of the first type instead of the Gegenbauer polynomials. These are given by $T_0(t) = 1$ and, for $n \in \mathbb{N}$,
\begin{equation}
\lim_{d \rightarrow 1^+}  \frac{2n+d-1}{2(d-1)} C_n^{\frac{d-1}{2}}(t) = T_n(t) = \cos( n \arccos(t)).
\end{equation}
Through this limit, one can quickly find that Lemmas \ref{lem:Riesz epsilon Gegen Coeff} and \ref{lem:Riesz Gegen Coeff} as well as Corollaries \ref{cor:Riesz Gegen Coeff} and \ref{cor:Riesz epsilon Gegen Coeff} still hold in this case.

\subsection{Limits for polarization}

We define, for any symmetric, lower semi-continuous kernel $K: \mathbb{S}^d \times \mathbb{S}^d \rightarrow(- \infty, \infty]$ and finite Borel measure $\mu$, the potential
\begin{equation}\label{eq:Potential Def}
U_{K}^{\mu}(x) = \int_{\mathbb{S}^d} K(x,y) d \mu(y).
\end{equation}
We define the polarization of $\mu$ to be
\begin{equation}
P_K(\mu) = \inf_{x \in \mathbb{S}^d} U_K^{\mu}(x).
\end{equation}

\begin{lemma}\label{lem:Limit for Polarization}
Let $K: \mathbb{S}^d \times \mathbb{S}^d \rightarrow(- \infty, \infty]$ be a symmetric, lower semi-continuous kernel. Let $\mu$ be a finite Borel measure on $\Omega$ and $(K_m)_{m=1}^{\infty}$ be a sequence of symmetric, continuous kernels, increasing pointwise in $m$ to $K$. Then
\begin{equation}
P_K(\mu) = \lim_{m \rightarrow \infty} P_{K_m}(\mu).
\end{equation}
\end{lemma} 

\begin{proof}
By the Monotone Convergence Theorem, we have that for all $x \in \mathbb{S}^d$,
\begin{equation*}
\lim_{m \rightarrow \infty} U_{K_m}^{\mu}(x) = U_K^{\mu}(x).
\end{equation*}

For $m \in \mathbb{N}$, let $x_m = \argmin U_{K_m}^{\mu}(x) $ and let $x_{\infty} = \argmin U_K^{\mu}(x)$. Thus
\begin{equation}\label{eq:lower bound}
\lim_{m \rightarrow \infty} U_{K_m}^{\mu}(x_m) \leq \lim_{m \rightarrow \infty} U_{K_m}^{\mu}(x_{\infty}) = U_K^{\mu}(x_{\infty}).
\end{equation}

For all $m \in \mathbb{N}$, we see that
\begin{equation*}
U_{K_m}^{\mu}(x_m)  \leq U_{K_m}^{\mu}(x_{m+1}) \leq  U_{K_{m+1}}^{\mu}(x_{m+1})    \leq  U_{K_{m+1}}^{\mu}(x_{\infty}) \leq U_{K}^{\mu}(x_{\infty}),
\end{equation*}
so $U_{K_m}^{\mu}(x_m)$ is an increasing sequence, bounded from above by $U_{K}^{\mu}(x_{\infty})$.

Since $\mathbb{S}^d$ is compact, there is a convergent subsequence $x_{m_k}$ of $x_m$, with limit point $x^* \in \mathbb{S}^d$. Now, for each $j, k, l \in \mathbb{N}$ such that $j \leq k \leq l$ we know
\begin{equation*}
U_{K_{m_j}}^{\mu}(x_{m_{k}}) \leq U_{K_{m_k}}^{\mu}(x_{m_{k}}) \leq U_{K_{m_k}}^{\mu}(x_{m_{l}}) \leq U_{K_{m_{l}}}^{\mu}(x_{m_{l}})
\end{equation*} 
so for all $j\leq k$
\begin{equation*}
U_{K_{m_j}}^{\mu}(x_{m_{j}}) \leq U_{K_{m_j}}^{\mu}(x_{m_{k}}) \leq \lim_{l \rightarrow \infty} U_{K_{m_l}}^{\mu}(x_{m_l}) = \lim_{m \rightarrow \infty}  U_{K_{m}}^{\mu}(x_{m}).
\end{equation*}
Thus, by continuity, for $j \in \mathbb{N}$,
\begin{equation*}
U_{K_{m_j}}^{\mu}( x^*) = \lim_{k \rightarrow \infty} U_{K_{m_j}}^{\mu}(x_{m_{k}})  \leq \lim_{m \rightarrow \infty}  U_{K_{m}}^{\mu}(x_{m}).
\end{equation*}
Thus, by the Monotone Convergence Theorem, we have 
\begin{equation}\label{eq:upper bound}
U_K^{\mu}(x^*) =\lim_{j \rightarrow  \infty} U_{K_{m_j}}(x^*)  \leq  \lim_{m \rightarrow \infty}  U_{K_{m}}^{\mu}(x_{m}).
\end{equation}

Our claim now follows from \eqref{eq:lower bound}, \eqref{eq:upper bound}, and the fact that $ U_K^{\mu}(x_{\infty}) \leq U_K^{\mu}(x^*)$.
\end{proof}

Taking $\mu = \sum_{j=1}^{N} \delta_{x_j}$ gives a discrete version of the result. We are unaware of a proof like this for any other polarization problems, and would like to point out that it should hold on arbitrary compact metric measure spaces as well.

\section{Acknowledgements}

D. Bilyk and M. Mastrianni have been supported by the NSF grant DMS-2054606. R.W. Matzke was supported by the Austrian Science Fund FWF project F5503 part of the Special Research Program (SFB) ``Quasi-Monte Carlo Methods: Theory and Applications" and NSF Postdoctoral Fellowship Grant 2202877. S. Steinerberger is supported by the NSF (DMS-2123224) and the Alfred P. Sloan Foundation.


\begin{thebibliography}{10}
\bibitem{ABD} C. Aistleitner, J.S. Brauchart, J. Dick. \emph{Point Sets on the Sphere $\mathbb{S}^2$ with Small Spherical Cap Discrepancy}. Discrete Comput. Geom. {\bf{48}}, 990--1024  (2012).

\bibitem{AmbBE} G. Ambrus, K.M. Ball, T. Erdelyi. \emph{Chebyshev constants for the unit circle}. Bull. Lond. Math. Soc. {\bf{45}}(2), 236--248 (2013).

\bibitem{AZ} K. Alishahi, M. Zamani. \emph{The spherical ensemble and uniform distribution of points on the sphere}. Electron. J. Probab. {\bf{20}}(23), 1--22 (2015).

\bibitem{ABS} D. Armentano, C. Beltr\'an, M. Shub. \emph{Minimizing the discrete logarithmic energy on the sphere: the role of random polynomials}. Trans. Amer. Math. Soc. {\bf{363}}(6), 2955--2965 (2011).

\bibitem{BagCR} J. Baglama, D. Calvetti, L. Reichel. \emph{Fast Leja points}. Electron. Trans. Numer. Anal. {\bf{7}}, 124--140 (1998).

\bibitem{beck1} J. Beck. \emph{Sums of distances between points on a sphere-- an application of the theory of irregularities of distribution to discrete geometry}. Mathematika {\bf{31}}, 33--41 (1984).


\bibitem{beck2} J. Beck. \emph{Some upper bounds in the theory of irregularities of distribution}. Acta Arith. {\bf{43}}, 115--130 (1984).

\bibitem{bellhouse} D. Bellhouse. \emph{Area estimation by point-counting techniques}. Biometrics {\bf{37}} (2), 303--312 (1981).

\bibitem{BE} C. Beltr\'an,  U. Etayo. \emph{The Diamond ensemble: A constructive set of spherical points with small logarithmic energy}. J. Complexity {\bf{59}}, 101471 (2020).

\bibitem{BelMO} C. Beltr\'an, J. Marzo, J. Ortega-Cerd\`a. \emph{Energy and discrepancy of rotationally invariant determinantal point processes in high dimensional spheres.} J. Complexity \textbf{37}, 76--109 (2016).

\bibitem{BetS} L. B\'{e}termin, E. Sandier. \emph{Renormalized energy and asymptotic expansion of optimal logarithmic energy on the sphere}. Constr. Approx. \textbf{47}(1), 39--74 (2018).

\bibitem{BiaCC} L. Bialas-Ciez, J.P. Calvi. \emph{Pseudo Leja sequences}. Ann. Mat. Pura Appl. {\bf{191}}, 53--75 (2012).

\bibitem{BD} D. Bilyk, F. Dai. \emph{Geodesic distance Riesz energy on the sphere}. Trans. Amer. Math. Soc. {\bf{372}}, 3141--3166 (2019).

\bibitem{BDM} D. Bilyk, F. Dai, R. Matzke. {\emph{Stolarsky principle and energy optimization on the sphere}}. Constr. Approx. {\bf{48}}(1), 31-60 (2018).

\bibitem{BMV} D. Bilyk, R. Matzke, O. Vlasiuk.  {\emph{Positive definiteness and the Stolarsky invariance principle}}. J. Math. Anal. Appl.  {\bf{513}}(2), Paper No. 126220, 30 pp. (2022).

\bibitem{Bj} G. Bj\"{o}rck. \emph{Distributions of positive mass, which maximize a certain generalized energy integral.} Arkiv F\"{o}r Matematik {\bf{3}}, 255--269 (1956).

\bibitem{BBL} E. Bogomolny, O. Bohigas, P. Leboeuf. \emph{ Distribution of roots of random polynomials}. Phys. Rev. Lett. {\bf{68}}, 2726--2729 (1992).


\bibitem{BorB} S.V. Borodachov, N. Bosuwan. \emph{Asymptotics of discrete Riesz $d$-polarization on subsets of $d$-dimensional manifolds}. Potential Anal. {\bf{41}}(1), 35--49 (2014).

\bibitem{BorHRS} S.V. Borodachov, D.P. Hardin, A. Reznikov, E.B. Saff. \emph{Optimal discrete measures for Riesz potentials}. Trans. Amer. Math. Soc. {\bf{370}} (10), 6973--6993 (2018).

\bibitem{BHS} S.V. Borodachov, D.P. Hardin, E.B. Saff. \emph{Discrete Energy on Rectifiable Sets}. Springer Monographs in Mathematics, Springer-Verlag New York (2019).

\bibitem{BosMSV} L. Bos, S. De Marchi, A. Sommariva, M. Vianello. \emph{Computing multivariate Fekete and Leja points by numerical linear algebra}, SIAM J. Numer. Anal. \textbf{48}, 1984-1999 (2010).



\bibitem{BDHSS22} P.G. Boyvalenkov, P.D. Dragnev, D.P. Hardin, E.B. Saff, M.M. Stoyanova. \emph{On Polarization of Spherical Codes and Designs}. Preprint, Arxiv:2207.08807  (2022).



\bibitem{Brauchart} J. Brauchart. \emph{About the second term of the asymptotics for optimal Riesz energy on the sphere in the potential-theoretical case}. Integral Transforms Spec. Funct. {\bf 17}(5), 321--328 (2006).

\bibitem{BraD} J.S. Brauchart, J. Dick. \emph{A Characterization of Sobolev Spaces on the Sphere and an Extension of Stolarsky's Invariance Principle to Arbitrary Smoothness}. Constr. Approx. \textbf{38}, 397--445 (2013).

\bibitem{BraG} J.S. Brauchart, P.J. Grabner. \emph{Distributing many points on spheres: minimal energy and designs}. J. Complexity \textbf{31}(3), 293--326 (2015).

\bibitem{BHS09} J.S. Brauchart, D.P. Hardin, E.B. Saff, \emph{The Riesz energy of the Nth roots of unity: an asymptotic expansion for large N}. Bull. Lond. Math. Soc. {\bf{41}}(4), 621-633 (2009).

\bibitem{BHS12} J.S. Brauchart, D.P. Hardin, E.B. Saff. \emph{The next-order term for optimal Riesz and logarithmic energy asymptotics on the sphere}. Recent Advances in Orthogonal Polynomials, Special Functions, and Their Applications. Contemp. Math. {\bf 578}, 31--61 (2012).

\bibitem{BRSSWW} J.S. Brauchart, A.B. Reznikov, E.B. Saff, I.H. Sloan, Y.G. Wang, R.S. Womersley. \emph{Random Point Sets on the Sphere-Hole Radii, Covering, and Separation}. Exp. Math. {\bf{27}}, 62--81 (2018).

\bibitem{brown} L. Brown, S. Steinerberger, \emph{Positive-definite Functions, Exponential Sums and the Greedy Algorithm: a curious Phenomenon}. J. Complexity {\bf{61}}, 101485 (2020).

\bibitem{CalVM11} J.P. Calvi, P. Van Manh. \emph{On the Lebesgue constant of Leja sequences for the unit disk and its applications to multivariate interpolation}. J. Approx. Theory {\bf{163}}, 608--622 (2011).

\bibitem{CalVM12} J.P. Calvi, P. Van Manh. \emph{Lagrange interpolation at real projections of Leja sequences for the unit disk}. Proc. Amer. Math. Soc. {\bf{140}}, 4271--4284 (2012).

\bibitem{CF} H. Chaix, H. Faure. \emph{Discr\'{e}pance et diaphonie en dimension un.} Acta Arith. {\bf{63}}, 103--141 (1993). (In French).

\bibitem{Chk} M.A. Chkifa. \emph{On the Lebesgue constant of Leja sequences for the complex unit disk and of their real projection}. J. Approx. Theory {\bf{166}}, 176--200 (2013).

\bibitem{CorD} C. Coroian, P. Dragnev. \emph{Constrained Leja points and the numerical solution of the constrained energy problem}. J. Comput. Appl. Math. {\bf{131}}, 427--444 (2001).

\bibitem{vdc} J.G. van der Corput. {\emph{Verteilungsfunktionen (Erste Mitteilung)}}. Proc. Sect. Sci. K. Ned. Akad. Wet. Amst. {\bf{38}}, 813--821 (1935). (In German).

\bibitem{DX} F. Dai, Y. Xu. \emph{Approximation Theory and Harmonic Analysis on Spheres and Balls}. Springer Monographs in Mathematics, Springer, New York, NY (2013).

\bibitem{DG} S.B. Damelin, P.J. Grabner. {\emph{Energy functionals, numerical integration and asymptotic equidistribution on the sphere}}. J. Complexity {\bf{19}}, 231--246 (2003).

\bibitem{DeM} S. De Marchi. \emph{On Leja sequences: some results and applications}. Appl. Math. Comput. {\bf{152}}, 621--647 (2004).

\bibitem{Erd} A. Edrei. \emph{Sur les d\'{e}terminants r\'{e}currents et les singularit\'{e}s d'une fonction donn\'{e}e par son d\'{e}veloppement de Taylor}. Compos. Math. {\bf{7}}, 20--88 (1940). (In French).

\bibitem{ErdS} T. Erd{\'e}lyi, E.B. Saff. \emph{Riesz polarization inequalities in higher dimensions}. J. Approx. Theory {\bf{171}}, 128--147 (2013).

\bibitem{etayo} U. Etayo. \emph{Spherical Cap Discrepancy of the Diamond Ensemble}. Discrete Comput. Geom. {\bf{66}}, 1218--1238 (2021).

\bibitem{FarN} B. Farkas, B. Nagy. \emph{Transfinite diameter, Chebyshev constant and energy on locally compact spaces}. Potential Anal. {\bf{28}}, 241--260 (2008).

\bibitem{FarR} B. Farkas, S. R{\'e}v{\'e}sz. \emph{Potential theoretic approach to rendezvous numbers}. Monatsh. Math. {\bf{148}}(4), 309--331 (2006).

\bibitem{fer} D. Ferizovi\'c. \emph{Spherical cap discrepancy of perturbed lattices under the Lambert projection}. Preprint, ArXiv:2202.13894 (2022).

\bibitem{FHM} D. Ferizovi\'c, J. Hofstadler, M. Mastrianni. \emph{The spherical cap discrepancy of HEALPix points} Preprint, ArXiv:2203.07552  (2022).

\bibitem{Gor} J. G\'orski. \emph{Les suites de points extr\'{e}maux li\'{e}s aux ensembles dans l’espace \`{a} 3 dimensions}. Ann. Polon. Math. {\bf{4}}, 14--20  (1957). (In French).

\bibitem{HEAL} K.M. G{\'o}rski, E. Hivon, A.J.  Banday, B.D. Wandelt, F.K. Hansen, M. Reinecke, M. Bartelmann. \emph{HEALPix: A Framework for High-Resolution Discretization and Fast Analysis of Data Distributed on the Sphere}. Astrophys. J. {\bf{622}}, 759-771 (2005).



\bibitem{Got01} M. G\"{o}tz. \emph{On the distribution of Leja-G\'{o}rski points}. J. Comp. Anal. App. {\bf{3}}, 223--241 (2001).

\bibitem{GraS} P. Grabner, T. Stepaniuk. \emph{Comparison of probablistic and deterministic point sets}.  J. Approx. Theory \textbf{239}, 128--143 (2019).

\bibitem{HarPS} D.P. Hardin, M. Petrache, E.B. Saff. \emph{Unconstrained polarization (Chebyshev) problems: basic properties and Riesz kernel asymptotics}. Potential Anal. {\bf{56}}, 21--64 (2022).

\bibitem{HKS13} D.P. Hardin, A.P. Kendall,  E.B. Saff. \emph{Polarization optimality of equally spaced points on the circle for discrete potentials}. Discrete Comput. Geom. {\bf{50}}, 236--243 (2013).

\bibitem{HarRSV} D.P. Hardin, A. Reznikov, E.B. Saff, A. Volberg. \emph{Local properties of Riesz minimal energy configurations and equilibrium measures}. Int. Math. Res. Not. IMRN {\bf{16}}, 5066--5086 (2019).

\bibitem{HarS} D.P. Hardin, E.B. Saff. \emph{Minimal Riesz energy point configurations for rectifiable $d$-dimensional manifolds}. Adv. Math. \textbf{193}(1), 174--204 (2005).

\bibitem{JanWZ} P. Jantsch, C.G. Webster, G. Zhang. \emph{On the Lebesgue constant of weighted Leja points for Lagrange interpolation on unbounded domains}. IMA J. Numer. Anal. {\bf{39}}, 1039--1057 (2019).

\bibitem{Kirk} N. Kirk. {\emph{On Proinov’s lower bound for the diaphony}}. Uniform Distribution Theory {\bf{15}}(2), 39--72 (2020).

\bibitem{kor} J. Korevaar. \emph{Fekete extreme points and related problems}. Approximation theory and function series (Budapest),  Bolyai Soc. Math. Stud. \textbf{5}, 35--62 (1995).

\bibitem{kritzinger} R. Kritzinger. \emph{Uniformly distributed sequences generated by a greedy minimization of the $L_2$ discrepancy}. Moscow J. Comb. Number Th. {\bf{11}}, 215--236 (2022).

 \bibitem{KS} A.B.J. Kuijlaars, E.B. Saff. \emph{Asymptotics for minimal discrete energy on the sphere}. Trans. Amer. Math. Soc. {\bf 350}(2), 523--538 (1998).

\bibitem{KN} L. Kuipers, H.  Niederreiter. \emph{Uniform Distribution of Sequences.} John Wiley (1974).

\bibitem{Lej} F. Leja. \emph{Sur certaines suites li\'{e}es aux ensembles plans et leur application \`{a} la repr\'{e}sentation conforme}. Ann. Polon. Math. {\bf{4}}, 8--13 (1957). (In French)



\bibitem{Lop} A. L{\'o}pez-Garc{\'i}a. \emph{Greedy energy points with external fields}. Contemp. Math. {\bf{507}}, 189--207 (2010).

\bibitem{LM21} A. L{\'o}pez-Garc{\'i}a, R. McCleary. \emph{Asymptotics of the minimum values of Riesz and logarithmic potentials generated by greedy energy sequences on the unit circle}. J. Math. Anal. Appl. {\bf{508}}, 125866 (2022).

\bibitem{LopM} A. L{\'o}pez-Garc{\'i}a, R. McCleary. \emph{Asymptotics of greedy energy sequences on the unit circle and the sphere}. J. Math. Anal. Appl. {\bf{504}}, 125269 (2021).

\bibitem{LopS} A. L{\'o}pez-Garc{\'i}a, E.B. Saff. \emph{Asymptotics of greedy energy points}. Math. Comp. {\bf{79}}, 2287--2316 (2010).

\bibitem{LopW} A. L{\'o}pez-Garc{\'i}a, D. Wagner. \emph{Asymptotics of the energy of sections of greedy energy sequences on the unit circle, and some conjectures for general sequences}. Comput. Methods Funct. Theory {\bf{15}}, 721--750 (2015).

\bibitem{alex} A. Lubotzky, R. Phillips,  P. Sarnak. \emph{Hecke operators and distributing points on the sphere. I.} Comm. Pure Appl. Math. {\bf{39}}, S149--S186 (1986).

\bibitem{marzo} J. Marzo, A. Mas. \emph{Discrepancy of minimal Riesz energy points}. Constr. Approx. {\bf{54}}, 473--506  (2021).

\bibitem{NarJ} A. Narayan, J.D. Jakeman. \emph{Adaptive Leja sparse grid constructions for stochastic collocation and high-dimensional approximation}. SIAM J. Sci. Comput. {\bf{36}}, A2952--A2983 (2014).

\bibitem{narco} F.J. Narcowich, X. Sun, J.D. Ward, Z. Wu. \emph{Leveque type inequalities and discrepancy estimates for minimal energy configurations on spheres}. J. Approx. Theory {\bf{162}}(6), 1256--1278 (2010).

\bibitem{Oht} M. Ohtsuka. \emph{On various definitions of capacity and related notions}. Nagoya Math. J.  {\bf{30}}, 121--127 (1967).

\bibitem{P20} F. Pausinger. \emph{Greedy energy minimization can count in binary: point charges and the van der Corput sequence}. Ann. di Mat. Pura ed Appl. {\bf{200}}, 165--186 (2021).

\bibitem{Pil} F. Pillichshammer. \emph{Optimal $L_2$-discrepancy and diaphony bounds for higher order digital sequences}. Preprint, ArXiv:2212.05747 (2022).

\bibitem{Pri1} I.E. Pritsker. \emph{Equidistribution of points via energy}. Ark. Mat. {\bf{49}}, 149--173 (2011).

\bibitem{Pri2} I.E. Pritsker. \emph{Distribution of point charges with small discrete energy}. Proc. Amer. Math. Soc. {\bf{139}}, 3461--3473 (2011).

\bibitem{Proi} P.D. Proinov. {\emph{Quantitative Theory of Uniform Distribution and Integral Approximation}}. University of Plovdiv, Bulgaria (2000). (In Bulgarian) 

\bibitem{PA} P.D. Proinov, E.Y. Atanassov. \emph{On the distribution of the van der Corput generalized sequences}C. R. Acad. Sci. Paris S\'{e}r. I Math. {\bf{307}}, 895--900 (1988).

\bibitem{RSZ} E.A. Rakhmanov, E.B. Saff, Y.M. Zhou. \emph{Minimal discrete energy on the sphere}. Math. Res. Lett. {\bf 1}(6), 647--662 (1994).

\bibitem{Rei} L. Reichel. \emph{Newton interpolation at Leja points}. BIT {\bf{30}}, 332--346 (1990).

\bibitem{RezSVl} A. Reznikov, E.B. Saff, O. Vlasiuk. \emph{A minimal principle for potentials with application to Chebyshev constants}. Potential Anal. {\bf{47}}, 235--244 (2017).

\bibitem{RezSVo} A. Reznikov, E.B. Saff, A. Volberg. \emph{Covering and separation of Chebyshev points for non-integrable Riesz potentials}. J. Complexity {\bf{46}}, 19--44 (2018).

\bibitem{Roth} K.F. Roth. {\emph{On irregularities of distribution}}. Mathematika {\bf{1}}, 73-79 (1954).

\bibitem{SafT} E.B. Saff, V. Totik. \emph{Logarithmic Potentials with External Fields}. Grundlehren Math. Wiss. {\bf{316}}, Springer (1997).

\bibitem{SanKH} G. Santin, T. Karvonen, B. Haasdonk. \emph{Sampling based approximation of linear functionals in reproducing kernel Hilbert spaces}. BIT Num. Math., \textbf{62}, 279--310 (2022).

\bibitem{Sic} J. Siciak. \emph{Two criteria for the continuity of the equilibrium Riesz potentials}. Ann. Soc. Math. Polonae, {\bf{14}}, 91--99 (1970).

\bibitem{Sim} B. Simanek. \emph{Asymptotically optimal configurations for Chebyshev constants with an integrable kernel}. N. Y. J. Math. {\bf{22}}, 667--675 (2016).

\bibitem{s} P. Sj\"ogren. \emph{Estimates of mass distributions from their potentials and energies}. Ark. Mat. {\bf{10}}, 59--77 (1972).

\bibitem{stein1} S. Steinerberger. \emph{Dynamically Defined Sequences with Small Discrepancy}. Monatshefte Math {\bf{191}}, 639--655 (2020).

\bibitem{stein2} S. Steinerberger. \emph{A Nonlocal Functional Promoting Low-Discrepancy Point Sets}. J. Complexity {\bf{54}}, 101410 (2019).

\bibitem{stein3} S. Steinerberger. \emph{Polynomials with Roots on the Unit Circle: Regularity of Leja sequences}. Mathematika {\bf{67}}, 553--268 (2021).

\bibitem{stein4} S. Steinerberger. \emph{On Combinatorial Properties of Greedy Wasserstein Minimization}. Preprint, ArXiv:2207.08043 (2022).



\bibitem{stol} K.B. Stolarsky. \emph{Sums of distances between points on a sphere. II.} Proc. Amer. Math. Soc. {\bf{41}}, 575--582 (1973).

\bibitem{TayT} R. Taylor, V. Totik. \emph{Lebesgue constants for Leja points}. IMA J. Numer. Anal. {\bf{30}}, 462--486 (2010).




\bibitem{wagner1} G. Wagner. \emph{On means of distances on the surface of a sphere (lower bounds)}. Pacific J. Math. {\bf 144} (2), 389--398 (1990).

\bibitem{wagner2} G. Wagner. \emph{On means of distances on the surface of a sphere. II. Upper bounds.} Pacific J. Math. {\bf 154} (2) 381--396 (1992).


\bibitem{wolf} R. Wolf. \emph{On the average distance property and certain energy integrals}. Ark. Mat. {\bf 35}, 387--400 (1997).


\end{thebibliography}
\end{document}